\documentclass[a4paper,11pt]{article}

\usepackage[a4paper, total={6.4in, 9.5in}]{geometry}
\usepackage{amsmath, amssymb, amsthm,url} 
\usepackage{graphicx,subfig}
\usepackage{caption}
\usepackage{color}

\newtheorem{alg}{Algorithm}
\newtheorem{thm}{Theorem}
\newtheorem{defn}{Definition}
\newtheorem{lem}{Lemma}
\newtheorem{prop}{Proposition}

\newtheorem{remark}{Remark}

\newtheorem{problem}{Problem}
\newcommand{\unit}{\texttt{1}\!\!\texttt{l}}
\newcommand{\M}{\mathcal{M}}
\newcommand{\N}{\mathcal{N}}
\newcommand{\grad}{\nabla_\Gamma}
\newcommand{\gradh}{\nabla_{\Gamma_h}}
\def \D #1{\underline{D}_{#1}}

\makeatletter
\let\gplgaddtomacro\g@addto@macro
\gdef\gplbacktext{}%
\gdef\gplfronttext{}%
\makeatother

\numberwithin{equation}{section}

\author{Hans Fritz\footnotemark[1]}

\title{On the computation of harmonic maps by unconstrained algorithms based on totally geodesic embeddings}

\date{}

\begin{document}
\maketitle
\renewcommand{\thefootnote}{\fnsymbol{footnote}}
\footnotetext[1]{Fakult\"at f\"ur Mathematik, Universit\"at Regensburg, 93040 Regensburg, Germany.\\ 
Hans.Fritz@ur.de}
\abstract{
In this paper, we present an algorithm for the computation of harmonic maps, and respectively, of the harmonic map heat flow between two closed Riemannian manifolds. 
Our approach is based on the totally geodesic embedding of the target manifold into $\mathbb{R}^N$. 
Since embeddings of Riemannian manifolds into Euclidean spaces can easily be made totally geodesic by extending the Riemannian metric in a certain way into some tubular neighbourhood,
the here presented approach is quite general. Totally geodesic embeddings allow to reformulate the harmonic map heat flow in a neighbourhood of the embedded target manifold.
This reformulation has the advantage that the problem becomes unconstrained: Instead of assuming a priori that the solution to the flow maps into the target manifold
this fact becomes a property of the solution to the extended flow for special initial data.
The solution space to the reformulated problem therefore exists of maps which are also allowed to map into the ambient space of the target manifold. 
This simplifies the discretization of the problem. 
Based on this observation, we here propose algorithms for the computation of the harmonic map heat flow and of harmonic maps.
In contrast to previous schemes, our algorithm does not make use of projections onto the target manifold, discrete tangential deformations, geodesic finite elements or of Lagrange multipliers.
We prove error estimates in the stationary case and present some numerical tests at the end of the paper.
}\\

\\

\\
\section*{Introduction}
\label{section_introduction}
The harmonic map equation is one of the most fundamental PDEs in mathematics, since it generalizes Laplace's equation to mappings between Riemannian manifolds.
Since the fundamental work of Eells and Sampson in \cite{ES64}, the study of harmonic maps has therefore become an important field of research.
From an analytic point of view the harmonic map equation and respectively, the associated flow, which is called the harmonic map heat flow, are demanding because of the 
non-linearity of the PDE and the constraint that the solution has to map into the target manifold. This condition implies that the set of admissible maps is in general not a linear space. 
It is clear that the polynomial interpolation of points lying on an embedded target manifold in $\mathbb{R}^N$ is, in general, not contained in the target manifold.
Therefore, a discretization of the problem based on standard finite elements would violate the above constraint. 
This is the reason why the development of algorithms for the computation of harmonic maps becomes quite tricky.
We here list different approaches which have been established in recent years to tackle this problem.
\begin{enumerate}
\item[1.] The condition of mapping into the target manifold can be implemented by using non-standard finite elements, the so-called geodesic finite elements, see \cite{Sa12,Sa15}.
\item[2.] In the following two approaches the constraint is weakened to finitely many points.
\begin{itemize}
\item[a)] In \cite{CD03,St14}, the discrete harmonic map is only assumed to have values 
in the target manifold for all mesh vertices. The problem of computing discrete harmonic maps is formulated using Lagrange multipliers, which leads to saddle point problems.
\item[b)] In \cite{Ba08, Ba10, Ba15}, harmonic maps are computed by solving the harmonic map heat flow. Although different schemes are proposed, the rough idea is always to deform
		a discrete map with values in the target manifold for all mesh vertices into the \textit{tangent} direction. The condition of the deformation being 
		tangent to the target manifold is again only imposed in the mesh vertices. This idea can be coupled with a back projection of the nodal values 
		onto the target manifold after each time step, see \cite{Ba08, Ba15} for details.  
\end{itemize}
For spherical target manifolds, numerical methods based on renormalizations, that is replacing a solution $u_{n+1}^\ast$ by $u_{n+1} = u_{n+1}^\ast/|u_{n+1}^\ast|$, 
were already developed in the 80s and 90s, see e.g. \cite{Al97, Ba05, CHKL87}. A property of these approaches is that 
great care has to be taken in order to ensure that the Dirichlet energy still
decreases after the renormalization step. Another method for one- and two-dimensional spheres was introduced in \cite{COV04}. It relies on polar coordinates, that is 
computations are done in the parameter domain and the constraint is automatically satisfied by the parametrization. 
\end{enumerate}

\subsubsection*{Our approach}
None of these numerical methods mimic the proof of existence of solutions to the harmonic map heat flow presented by Hamilton in \cite{Ham75}.
This is a bit surprising, since this proof is based on a reformulation of the harmonic map heat flow, which seems to be advantageous for numerical purposes.
The idea of the proof is first to embed the target manifold into some Euclidean space. Due to the Nash embedding theorem this is always possible if the co-dimension is sufficiently high 
and the embedding can even be assumed to be isometric, which however is not crucial here. In the non-isometric case, a Whitney embedding would also be sufficient. 
In the second step, the harmonic map heat flow, which describes the evolution of a mapping between two Riemannian manifolds,
is reformulated as an evolution equation for a map with values in the ambient space of the target manifold.
By this means, the problem becomes unconstrained. A different theoretical method, where the original problem is replaced by an unconstrained problem,
can be found in \cite{CL95,St98}. There, the authors make use of a penalty method and consider the corresponding limit.
This approach, however, is totally different from the work in \cite{Ham75} and will not play any role in this paper.  
We will follow the ideas of Hamilton to develop numerical schemes for the computation of the harmonic map heat flow and of harmonic maps. 

For the sake of simplicity we first consider the harmonic map heat flow $f: \Gamma \times [0,T) \rightarrow  \M$
between two hypersurfaces $\Gamma \subset \mathbb{R}^{n+1}$ and $\M \subset \mathbb{R}^{n+1}$ of the same dimension.
It can be written in the following form
\begin{equation}
	\frac{\partial}{\partial t} f = \Delta_{\Gamma} f 
	+ \sum_{\alpha, \beta =1}^{n+1} (\mathcal{H}_{\alpha \beta} \nu)\circ f \nabla_{\Gamma} f^\alpha \cdot \nabla_{\Gamma} f^\beta.
	\label{HMF_embedded}
\end{equation}
Here, $\Delta_\Gamma$ and $\nabla_\Gamma$ denote the Laplace-Beltrami operator and respectively, the tangential gradient on $\Gamma$.
The vector field $\nu$ is supposed to be a unit normal field to $\M \subset \mathbb{R}^{n+1}$ and $\mathcal{H}$ is the associated extended Weingarten map
$\mathcal{H}= \nabla_{\M} \nu$. Using the orthogonal projection $a: \N_T \rightarrow \M$ onto the manifold $\M$, which is well-defined in 
a tubular neighbourhood $\N_T \subset \mathbb{R}^{n+1}$ of $\M$, a possible extension of the harmonic map heat flow would be the following problem:
Find $f_{ex}: \Gamma \times [0,T) \rightarrow \N_T$ such that
$$
	\frac{\partial}{\partial t} f_{ex} = \Delta_{\Gamma} f_{ex} + \sum_{\alpha, \beta =1}^{n+1} (\mathcal{H}_{\alpha \beta} \nu)\circ a \circ f_{ex} 
											\nabla_{\Gamma} f^\alpha_{ex} \cdot \nabla_{\Gamma} f^\beta_{ex},
$$ 
with $f_{ex}(\cdot, 0) = f_0(\cdot)$ for some suitable initial data $f_0: \Gamma \rightarrow \M$. A solution to the harmonic map heat flow (\ref{HMF_embedded}) would clearly also satisfy
this extended equation. However, the following argument shows that this extension is not suitable for a numerical method. Let $\Gamma$ and $\M$ be the round unit sphere
in $\mathbb{R}^{n+1}$. Then the identity map is a harmonic map between $\Gamma$ and $\M$. We now consider perturbations obtained by a uniform
scaling of the identity map by a factor of $r(t): [0,T) \rightarrow (0,\infty)$. A short calculation shows that the extended harmonic map heat flow leads to the following ODE
$$
	r'(t) = - n r(t) + n r(t)^2.
$$   
A solution $r(t)$ to this equation is monotonically increasing if $r(0) > 1$ and monotonically decreasing if $r(0) < 1$. We therefore expect that a numerical solution to the
above extension would be unstable under small perturbations. This shows that an extension of the harmonic map heat flow to a neighbourhood of the target manifold must be chosen very carefully.
However, as we will see below, there is a general method to find such a reformulation, which in the end will lead to nice algorithms. 
   
\subsubsection*{Related work}
Unconstrained numerical schemes for the harmonic map heat flow or in this case, more generally, for the $p$-harmonic map heat flow were also introduced in the paper of Osher
and Vese, see \cite{OV02}. For spherical target manifolds their idea is to consider the Dirichlet energy for the map $U = V / |V|$, where $V$ maps into the Euclidean space.
Obviously, this energy is invariant under scaling of $V$. They then derive the gradient flow, that is the evolution equations for the components of $V$. 
After discretization they obtain a numerical scheme, which preserves the property $|V_0| = 1$ in time. In particular, no renormalization has to be applied. 
The origin of this property is that the radial component of the variation of their new energy vanishes.   
Although the motivation behind the work in \cite{OV02} is similar to ours, their idea of getting rid of the constraint is totally different from ours. 
In particular, we will obtain a non-degenerate parabolic PDE-system.
We are not aware of any previous publications, where the theoretical idea to reformulate the harmonic map heat flow as an unconstrained problem by using totally geodesic 
embeddings was used to develop numerical schemes.


\subsubsection*{Outline of the paper}
In the next section, we introduce our notation and some basic results from differential geometry. In Section \ref{section_extended_HMF}, we proceed as follows.
We first introduce the harmonic map heat flow and respectively, the harmonic map equation
between two (not necessarily embedded) Riemannian manifolds. We then describe a method to construct a totally geodesic embedding
of the target manifold. Based on this construction, a suitable extension of the harmonic map heat flow can be formulated, see Section \ref{subsection_extension_HMF}.
The case of the target manifold being a round sphere is considered explicitly. 
In Section \ref{section_stability_HMF}, the stability of the extended flow is discussed, and 
in Section \ref{section_weak_formulation_HMF}, a weak formulation of the extended flow is derived. 
We then recall the surface finite element method in Section \ref{section_SFEM}.
Discretization of the weak formulation and the numerical analysis of our novel schemes are discussed in Sections \ref{section_discrete_problems} and \ref{section_numerical_analysis},
respectively. Details about the implementation are given in Section \ref{section_implementation}. In Section \ref{section_numerical_examples}, we present some numerical examples.

\section{Notation and Preliminaries}
\label{section_notation}
Throughout the paper $\Gamma$ and $\M$ will denote two smooth, not necessarily embedded manifolds. The dimensions of $\Gamma$ and $\M$ may differ if not otherwise stated.
$\Gamma$ and $\M$ are assumed to be closed, that is compact and without boundary.
Henceforward, we will make use of the convention to sum over repeated indices. To keep notation simple yet concise Roman indices will refer to the local coordinates 
of the smooth manifold $\Gamma$, whereas Greek indices will refer to the local coordinates of the smooth manifold $\M$ -- if not otherwise stated.
The Christoffel symbols, the gradient and the corresponding covariant derivative with respect to a Riemannian metric $m$ will be denoted by $\Gamma(m)^k_{ij}$,
$grad_m$ and $\nabla^m$. 
Partial derivatives in the parameter domain of a parametrization are denoted by $\partial_i$, whereas the partial derivatives in the ambient Euclidean space of
an embedding are denoted by $D_\alpha$, where we have used Greek indices to refer to the Euclidean coordinates of the ambient space.
The Laplacian with respect to $m$ of a twice-differentiable function $f$ is defined by 
$$
	(\Delta_m f) \circ \mathcal{C}^{-1} = m^{ij} \left( \partial_i \partial_j F - \Gamma(m)_{ij}^k \partial_k F \right).
$$
Here, $\mathcal{C}$ is a coordinate chart of the corresponding smooth manifold and $F := f \circ \mathcal{C}^{-1}$ is the local coordinate function for $f$
defined on some parameter domain $\Omega$. In the following, we will use (Roman and Greek) capital letters for the local coordinate functions 
of functions defined on a manifold if it is helpful. 
For a map $f: (\Gamma, m) \rightarrow (\M, g)$ between two Riemannian manifolds $(\Gamma, m)$ and $(\M,g)$ the map Laplacian is defined by
$$
	(\mathcal{C}_2 \circ (\Delta_{m,g} f) \circ \mathcal{C}^{-1}_1)^\alpha = m^{ij} \left(
		\partial_i \partial_j F^\alpha - \Gamma(m)_{ij}^l \partial_l F^\alpha + \Gamma(g)^\alpha_{\beta \gamma} \circ F \ \partial_i F^\beta \partial_j F^\gamma
	\right),
$$ 
where $\mathcal{C}_2$ is a coordinate chart of $\M$ and $\mathcal{C}_1$ is a coordinate chart of $\Gamma$. Henceforward, we will neglect the coordinate charts for the sake of simplicity.
The Riemannian volume form of a Riemannian manifold will be denoted by $do$. 

The Euclidean scalar product is denoted by $v \cdot w = \sum_i v_i w_i$ and the corresponding matrix scalar product by $A:B = \sum_{i,j} A_{ij} B_{ij}$.
The Euclidean norm is denoted by $|\cdot|$.
For an orientable smooth closed embedded hypersurface $\mathcal{S} \subset \mathbb{R}^{n+1}$ we define the signed (Euclidean) distance function $d$ to $\mathcal{S}$
by 
\begin{equation}
 d(x) := \left\{
 \begin{aligned}
 	&\inf_{p \in \M} |x - p| \quad \textnormal{if} \quad x \in \mathbb{R}^{n+1} \setminus U, \\
 	&- \inf_{p \in \M} |x - p| \quad \textnormal{else},
 \end{aligned}
 \right.
 \label{signed_distance_function}
\end{equation} 
where $U \subset \mathbb{R}^{n+1}$ is a bounded domain with $\partial U = \mathcal{S}$. For a smooth hypersurface $\mathcal{S} \subset \mathbb{R}^{n+1}$ (that means smooth as an embedding) 
the signed distance function is also smooth in some neighbourhood of $\mathcal{S}$ and we have
\begin{equation}
	|Dd(x)| = 1, \quad \textnormal{and hence} \quad D^2 d(x) Dd(x) = 0, 
	\label{properties_signed_distance_function}
\end{equation}
see for example \cite{DDE05}. 
The outward unit normal $\nu$ to $\mathcal{S}$ is defined by $\nu(x) := Dd(x)$ for $x \in \mathcal{S}$ and the tangential projection is $P := \unit - \nu \otimes \nu$.
The tangential gradient of a differentiable function $f:\mathcal{S} \rightarrow \mathbb{R}$ is
$\nabla_{\mathcal{S}} f := P D \tilde{f}$, where $\tilde{f}$ is a differentiable yet arbitrary extension of $f$ into a neighbourhood of $\mathcal{S}$.
It is easy to show that this definition only depends on the values of $f$ on $\mathcal{S}$. In fact, the tangential gradient is given by the
gradient $grad_m f$ if $m$ is the metric on $\mathcal{S}$ induced the Euclidean metric. The Laplace-Beltrami operator of a twice-differentiable function $f$
is $\Delta_{\mathcal{S}} f := \nabla_{\mathcal{S}} \cdot \nabla_{\mathcal{S}} f$. If $m$ is the induced metric, then $\Delta_{\mathcal{S}} f = \Delta_m f$.

\begin{prop}
\label{Prop_invariance_map_Laplacian}
Let $(\Gamma,m)$ and $(\M,g)$ be two smooth Riemannian manifolds. Furthermore, let $\phi: \M \rightarrow \M$ be a smooth diffeomorphism.
Then, the map Laplacian is invariant under $\phi$, that is
$$
	(d\phi \circ f) (\Delta_{m,g} f) = \Delta_{m, \phi_\ast g} (\phi \circ f)
$$
\end{prop}
\begin{proof} In local coordinates, we compute
\begin{align*}
	&(\partial_\alpha \Phi^\kappa) \circ F \ m^{ij} 
	(\partial_i \partial_j F^\alpha - \Gamma(m)^l_{ij} \partial_l F^\alpha + \Gamma(g)_{\beta \gamma}^\alpha \circ F \ \partial_i F^\beta \partial_j F^\gamma)
	\\
	&= m^{ij} 
		(\partial_i \partial_j (\Phi \circ F)^\kappa - (\partial_\beta \partial_\gamma \Phi^\kappa) \circ F \ \partial_i F^\beta \partial_j F^\gamma 
	\\	
	& \qquad \quad - \Gamma(m)^l_{ij} \partial_l (\Phi \circ F)^\kappa + (\partial_\alpha \Phi^\kappa \Gamma(g)_{\beta \gamma}^\alpha) \circ F \ \partial_i F^\beta \partial_j F^\gamma).
\end{align*}
The claim then follows from the fact that
$$
	\partial_\alpha \Phi^\kappa \Gamma(g)_{\beta \gamma}^\alpha - \partial_\beta \partial_\gamma \Phi^\kappa 
	= (\Gamma(\phi_\ast g)^\kappa_{\iota \rho}) \circ \Phi \ \partial_\beta \Phi^\iota \partial_\gamma \Phi^\rho, 
$$
which can be seen by a straightforward calculation using the definition of the Christoffel symbols and the definition of the push-forward metric $\hat{g}:= \phi_\ast g$, 
which means that 
$$
	g_{\alpha \beta} = \hat{g}_{\kappa \iota} \circ \Phi \ \partial_\alpha \Phi^\kappa \partial_\beta \Phi^\iota.
$$
\end{proof}
A map $i: \M \rightarrow \M$ is called an involution if $i(i(p)) = p$ for all $p \in \M$.
\begin{defn}
Let $(\M,g)$ be a Riemannian submanifold of $(\N, G)$. The tensor $II(\cdot, \cdot): T\M \times T\M \rightarrow N\M$ 
on the tangent bundle $T\M$ with image in the normal bundle $N\M$ defined by $II(u,w) := \nabla^G_u w - \nabla^g_u w$ for $u, w \in T\M$ is called the shape tensor of $\M$ in $\N$.
\end{defn}
\begin{defn}
A Riemannian submanifold $(\M,g)$ of a Riemannian manifold $(\N, G)$ is called totally geodesic if the shape tensor $II = 0$ of $\M$ in $\N$ vanishes.
\end{defn}
That such a submanifold is called totally geodesic is motivated by the following result.
\begin{prop}
A Riemannian submanifold $(\M,g)$ of a Riemannian manifold $(\N,G)$ is a totally geodesic submanifold if and only if
any geodesic on the submanifold $(\M,g)$ is also geodesic on the Riemanian manifold $(\N,G)$. 
\end{prop}
\begin{proof}
See \cite{ON83}, Chapter 4, Proposition 13.
\end{proof}

\section{The extended harmonic map heat flow}
\label{section_extended_HMF}

\subsection{Harmonic maps and the harmonic map heat flow}
\label{section_HMF}
\begin{defn}
Let $(\Gamma,m)$ and $(\M,g)$ be two smooth closed Riemannian manifolds. Furthermore, let $f_0: \Gamma \rightarrow \M$ be a smooth map.
A smooth solution $f: \Gamma \times [0,T) \rightarrow \M$ to the evolution equation
\begin{equation}
f_t = \Delta_{m,g} f \quad \textnormal{with $f(\cdot, 0) = f_0(\cdot)$}.
\label{harmonic_map_heat_flow}
\end{equation}
is called a harmonic map heat flow.
\end{defn}
\begin{remark}
The harmonic map heat flow is the $L^2$-gradient flow for the Dirichlet energy 
$$
	E(f):= \frac{1}{2} \int_\Gamma \mbox{trace}_m (f^\ast g) \, do.
$$
\end{remark}
\begin{remark}
Let $(\Gamma,m)$ and $(\M,g)$ be two $n$-dimensional closed Riemannian manifolds which are isometrically and smoothly embedded into $\mathbb{R}^{n+1}$.
Then $f: \Gamma \times [0,T) \rightarrow \M$ is a harmonic map heat flow if and only if it solves (\ref{HMF_embedded}).
\end{remark}
Short-time existence and uniqueness of harmonic map heat flows were proved in \cite{ES64}.
It is an interesting question whether the solution for the above flow exists for all times. In \cite{ES64} an affirmative answer was given for target
manifolds $(\M,g)$ of negative Riemannian curvature.
Using this result, it is possible to establish existence of stationary solutions (in each homotopy class of $f_0$) 
by considering the limit $t \rightarrow \infty$ of the harmonic map heat flow.
\begin{defn}
Let $(\Gamma,m)$ and $(\M, g)$ be two smooth closed Riemannian manifolds. A smooth solution $f: \Gamma \rightarrow \M$ to
\begin{equation}
\Delta_{m,g} f = 0
\end{equation}
is called a harmonic map.
\end{defn} 
In this paper, we aim to tackle the following problems. 
\begin{problem}
\label{Problem_HMF}
Let $(\Gamma,m)$ and $(\M, g)$ be two smooth closed Riemannian manifolds. Compute an approximation to the harmonic map heat flow (\ref{harmonic_map_heat_flow}).
\end{problem}
\begin{problem}
\label{Problem_harmonic_maps}
Let $(\Gamma,m)$ and $(\M,g)$ be two smooth closed Riemannian manifolds. Compute an approximation to a harmonic map $f: (\Gamma,m) \rightarrow (\mathcal{M},g)$.
\end{problem}
Our strategy to solve Problem \ref{Problem_harmonic_maps} numerically is to consider the long-time behaviour of solutions to Problem \ref{Problem_HMF}.
This approach is closely related to the existence proof of harmonic maps in \cite{ES64} and was already used in the works of Bartels, see e.g. \cite{Ba08},
who considers the corresponding $H^1$-gradient flow for stability reasons. The main difference between our work and the results of Bartels are therefore
the way how we solve Problem \ref{Problem_HMF}.
Instead of computing geometric flows, the works of Clarenz and Dziuk in \cite{CD03} as well as of Steinhilber in \cite{St14} rely on a Newton type method for a reformulated problem.
Although Newton's method usually converges much faster, its convergence in general depends on a good initial guess.

\subsection{Construction of a totally geodesic embedding.}
\label{section_totally_geodesic_embeddings}

The construction of a totally geodesic embedding relies on the following result. 
\begin{thm}
\label{Theorem_totally_geodesic}
Let $(\M,g)$ be a Riemannian submanifold of $(\N,G)$ and let $i: \N \rightarrow \N$ be an isometry, such that $\M$ is a (path-)connected component of the fixed point set of
$i$, that is of $\{ p \in \N: i(p) = p \}$. Then $\M$ is totally geodesic. 
\end{thm}
\begin{proof}
See \cite{Kl82}, Theorem $1.10.15$.
\end{proof}
An obvious non-trivial example for a totally geodesic submanifold is the $k$-sphere $S^k = \{ x \in \mathbb{R}^{n+1}: |x| = 1 \; \textnormal{and} \; 
x_{k+2} = \ldots x_{n+1} = 0 \}$ for $1 \leq k \leq n-1$ considered as a submanifold 
of the standard $n$-sphere $\mathbb{S}^n = \{ x \in \mathbb{R}^{n+1}: |x| = 1 \}$ with Riemannian metric induced by the Euclidean metric in $\mathbb{R}^{n+1}$.
This follows directly from the fact that $S^k \subset \mathbb{S}^n$ is the fixed point set of the following isometry $i: \mathbb{S}^n \rightarrow \mathbb{S}^n$ with
$x \mapsto (x_1, \ldots, x_{k+1}, -x_{k+2}, \ldots, -x_{n+1})$.

Henceforward, we will assume that the $n$-dimensional smooth manifold $\M$ is smoothly embedded into $\mathbb{R}^N$ for $N$ large enough.
From the strong Whitney theorem it follows that this assumption is no restriction 
since such an embedding always exists for $N=2n$. Please note that, in general, it is not necessary to assume that this embedding is isometric.
In order to make the Riemannian manifold $(\M,g)$ a totally geodesic submanifold of some other Riemannian manifold we need the following two ingredients
\begin{enumerate}
\item An extension of the Riemannian metric $g$ on $\M \subset \mathbb{R}^N$ to a Riemannian metric $G'$ on a suitable tubular neighbourhood $\N_T$ such that
$(\M, g)$ is a Riemannian submanifold of $(\N_T, G')$.
\item An involution $i: \N_T \rightarrow \N_T$ on the tubular neighbourhood $\N_T$ of $\M$ such that $\M$ is the fixed point set of $\M$.
\end{enumerate}
We then define the Riemannian metric $G$ on $\N_T$ by
\begin{equation}
	G := \tfrac{1}{2} (G' + i^\ast G').
	\label{averaged_metric}
\end{equation}
Since $\M$ is the fixed point set of $i$, we have $G = G'$ on $\M$. By assumption, $G$ therefore induces the Riemannian metric $g$ on $\M$.
Moreover, since $i^\ast G = \tfrac{1}{2} (i^\ast G' + i^\ast (i^\ast G')) = \tfrac{1}{2} (i^\ast G' + G') = G$, the involution $i$ is 
an isometry. Using Theorem \ref{Theorem_totally_geodesic}, we can conclude that $(\M,g)$ is totally geodesic submanifold of $(\N_T, G)$.
The idea of the above construction in the context of harmonic maps on manifolds with boundary can be found in \cite{Ham75}, Chapter IV.5.
A detailed description is also given in \cite{Hu94}, Chapter 4.1.
In the following, we describe how to find an involution $i$ satisfying the above condition in the most important case for applications, that is when $\M \subset \mathbb{R}^{n+1}$
is an isometrically embedded hypersurface.
We start by considering the special case of $\M = \mathbb{S}^n \subset \mathbb{R}^{n+1}$.

\subsubsection*{Example: The round unit sphere as a totally geodesic submanifold.}
In order to demonstrate the practicability of the above construction, we now consider the standard sphere $\mathbb{S}^n \subset \mathbb{R}^{n+1}$ with the metric induced 
by the Euclidean metric of the ambient space. The Euclidean metric is then an extension of the metric $g$ on $\M$ by construction. The sphere inversion
$i: \mathbb{R}^{n+1}\setminus \{ 0 \} \rightarrow \mathbb{R}^{n+1} \setminus \{ 0 \}$ with $x \mapsto x/|x|^2$ is an involution on the neighbourhood 
$\N_T := \mathbb{R}^{n+1}\setminus \{0\}$. Its derivative is given by
$$
	Di(x) = \tfrac{1}{|x|^2} \left( \unit - 2 \tfrac{x}{|x|} \otimes \tfrac{x}{|x|} \right) = \tfrac{1}{|x|^2} \left( \unit - 2 \nu(x) \otimes \nu(x) \right),
$$  
where $\nu(x) = x/|x|$ is the extension of the outward unit normal, which is constant in the normal direction.
We average the Euclidean metric under $i$ like in (\ref{averaged_metric}) and obtain
\begin{equation}
	G(x) = \tfrac{1}{2} \left( \unit + Di^T(x) Di(x) \right) = \left( \tfrac{1}{2} + \tfrac{1}{2|x|^4} \right) \unit.
	\label{extended_sphere_metric}
\end{equation}
We define $\rho: \mathbb{R}^{n+1}\setminus \{0\} \rightarrow \mathbb{R}$ by $\rho(x) := \tfrac{1}{2} + \tfrac{1}{2|x|^4}$ and write $G(x) = \rho(x) \unit$.
Please note that it is possible to alter $\rho$ in such a way that we obtain a smooth Riemannian metric $G$ on the whole of $\mathbb{R}^{n+1}$ 
without changing $G$ in a neighbourhood of $\M$. Then, $(\M,g)$ would be a totally geodesic submanifold of $(\mathbb{R}^{n+1}, G)$.
However, since this is only a technical step and since we will, in fact, not use the values of the metric in a neighbourhood of the origin,
we here ignore this point for the sake of simplicity.  
For non-spherical target manifolds, it will not be possible to use the sphere inversion as an involution. 
We will next show that there is a simple alternative for the construction of an involution, which is based on the usage of 
Fermi coordinates in a suitable neighbourhood of $\M$. 

\subsubsection*{Construction of an extension of the metric and of an involution}
We now consider the case when the target manifold $(\M,g)$ is an orientable, $n$-dimensional, smooth closed hypersurface in $\mathbb{R}^{n+1}$,
whose metric is induced by the Euclidean metric of the ambient space. Such a target manifold $(\M,g)$ will not be a totally geodesic submanifold of the ambient Euclidean space.
We will therefore now construct a metric $G$ on some tubular neighbourhood $\N_T$ of $\M$ such that $(\M,g)$ is a totally geodesic submanifold of $(\N_T,G)$.
Since $g$ is induced by the Euclidean metric, we choose $G'_{\alpha\beta} = \delta_{\alpha\beta}$. In order to construct an involution $i$ with fixed point set
$\M$, we next choose a tubular neighbourhood $\N_T$ of fixed width such that for all $x \in \N_T$ the decomposition $x = a(x) + d(x) Dd(x)$ with $a(x) \in \M$ is unique.
Here, $d: \N_T \rightarrow \mathbb{R}$ denotes the signed distance function (\ref{signed_distance_function}) to $\M \subset \mathbb{R}^{n+1}$.
On $\N_T$ we then consider the map $i : \N_T \rightarrow \N_T$ defined by 
\begin{equation}
	i(x) := x - 2 d(x) Dd(x).
	\label{involution_oriented_distance}
\end{equation}
Using the fact that $|Dd| = 1$ on $\N_T$, see (\ref{properties_signed_distance_function}), it is easy to see that $\M$ is the fixed point set of $i$.
Moreover, we have
\begin{align*}
 i(i(x)) &= x - 2 d(x) Dd(x) - 2 d(x- 2d(x) Dd(x)) Dd(x - 2 d(x) Dd(x)) \\
 &= x - 2 d(x) Dd(x) - 2 (d(x) - 2d(x)) Dd(x) = x,
\end{align*}
where we have used the fact that $d(x)$ is a linear function on the segment from $x$ to $x - 2d(x) Dd(x)$ and $|Dd| = 1$, see \cite{DDE05}.
It follows that $i$ is an involution on $\N_T$. The derivative of $i$ is given by	
$$
	Di(x) = \unit - 2 Dd(x) \otimes Dd(x) - 2 d(x) D^2d(x).
$$
Using (\ref{properties_signed_distance_function}), we obtain that
$$
	Di^T(x) Di(x) = \unit - 4 d(x) D^2d(x) + 4 d(x)^2 D^2d(x) D^2d(x).
$$
The Riemannian metric $G$ on $\N_T$ defined in (\ref{averaged_metric}) is therefore given by
\begin{equation}
	G(x) = \tfrac{1}{2} \left( \unit + Di^T(x) Di(x) \right) = \unit - 2d(x) D^2 d(x) + 2 d(x)^2 D^2d(x) D^2d(x).
	\label{extended_metric_hypersurfaces}
\end{equation}
For example, for the standard $n$-sphere in $\mathbb{R}^{n+1}$, the metric $G$ based on the involution (\ref{involution_oriented_distance}) 
is given by
$$
	G(x) = \unit + \tfrac{2}{|x|^2} (1 - |x|) \left( \unit - \tfrac{x}{|x|} \otimes \tfrac{x}{|x|} \right).
$$
If the target manifold is the standard sphere, we will only use the metric (\ref{extended_sphere_metric}) based on the sphere inversion in the following. 
So far, we have not treated the case when the target manifold is a submanifold of higher co-dimensions 
and when the metric $g$ is not induced by the Euclidean metric of the ambient space.
In the latter case, a starting point would be to use an extension $G'$ of $g$ like in \cite{Fr15} and generalized distance functions.
We leave this problem for future research.

\subsection{Extension of the harmonic map flow based on totally geodesic embeddings}
\label{subsection_extension_HMF}
\begin{thm}
\label{Theorem_extended_map_Laplacian}
Let $(\M,g)$ be a totally geodesic submanifold of $(\N,G)$ and $f: \Gamma \rightarrow \M \subset \N$ a $C^2$-map defined
on the Riemannian manifold $(\Gamma,m)$ then 
$$
	\Delta_{m,g} f = \Delta_{m,G} f.
$$
\end{thm}
\begin{proof}
The proof follows the argumentation in \cite{Ham75}.
Let $\dim \Gamma = d$, $\dim \M = n$ and $\dim \N = n+r$. By the definition of a submanifold, we can choose coordinates $\{x_1, \ldots, x_n, x_{n+1}, \ldots, x_{n+r} \}$
of $\N$ around $p \in \M$ such that we
locally have $\M = \{ q \in \N: x_{n+1} = \ldots = x_{n+r} = 0 \}$. Since $f: \Gamma \rightarrow \M$, we have
$F^{n+1} = \ldots = F^{n+r} = 0$ in these coordinates. 
\begin{align*}
\Delta_{m,G} f &= \sum_{\alpha=1}^{n+r} \sum_{i,j=1}^d m^{ij} \left( \partial_i \partial_j F^\alpha - \sum_{l=1}^d \Gamma(m)_{ij}^l \partial_l F^\alpha
					+ \sum_{\beta, \gamma=1}^{n+r} \Gamma(G)_{\beta \gamma}^\alpha \circ F \ \partial_i F^\beta \partial_j F^\gamma \right) \partial_\alpha
 \\
 &= \sum_{\alpha=1}^{n+r} \sum_{i,j=1}^d m^{ij} \left( \partial_i \partial_j F^\alpha - \sum_{l=1}^d \Gamma(m)_{ij}^l \partial_l F^\alpha
	+ \sum_{\beta,\gamma=1}^{n} \Gamma(G)_{\beta \gamma}^\alpha \circ F \ \partial_i F^\beta \partial_j F^\gamma \right) \partial_\alpha.
\end{align*}
Since $\M$ is a totally geodesic submanifold of $\N$, we conclude that
\begin{align*}
	0 &= II(u,w) \equiv \nabla^G_u w - \nabla^g_u w\\
	  &= \sum_{\alpha, \beta=1}^{n+r} U^\beta \left( \partial_\beta W^\alpha - \sum_{\gamma=1}^{n+r} \Gamma(G)_{\beta \gamma}^\alpha W^\gamma \right) \partial_\alpha
	   - \sum_{\alpha, \beta=1}^n U^\beta \left( \partial_\beta W^\alpha - \sum_{\gamma=1}^n \Gamma(g)_{\beta \gamma}^\alpha W^\gamma \right) \partial_\alpha
	  \\
	  &= - \sum_{\alpha, \beta, \gamma =1}^{n+r} \Gamma(G)_{\beta \gamma}^\alpha U^\beta W^\gamma \partial_\alpha
	   + \sum_{\alpha, \beta, \gamma =1}^n \Gamma(g)_{\beta \gamma}^\alpha U^\beta W^\gamma \partial_\alpha 
	  \\
	  &=  - \sum_{\alpha, \beta, \gamma =1}^{n} \left( \Gamma(G)_{\beta \gamma}^\alpha - \Gamma(g)_{\beta \gamma}^\alpha \right) U^\beta W^\gamma \partial_\alpha
	   - \sum_{\alpha=n+1}^{n+r} \sum_{\beta,\gamma = 1}^n \Gamma(G)_{\beta \gamma}^\alpha U^\beta W^\gamma \partial_\alpha 
\end{align*}
for all $C^1$-vector fields $u,v$  that are tangential to $\M$, that is $U^{n+1} = \ldots = U^{n+r} = 0$ and $W^{n+1} = \ldots = W^{n+r} = 0$.
Hence,
\begin{align*}
& \Gamma(G)_{\beta\gamma}^\alpha = \Gamma(g)_{\beta\gamma}^\alpha, \qquad \forall \alpha, \beta, \gamma \in \{1, \ldots, n\},
\\
& \Gamma(G)_{\beta\gamma}^\alpha = 0, \qquad \forall \beta, \gamma \in \{1, \ldots, n \} \ \textnormal{and} \ \alpha = n+1, \ldots, n+r. 
\end{align*}   
We conclude that
\begin{align*}
 \Delta_{m,G} f  &= \sum_{\alpha=1}^{n} \sum_{i,j=1}^d m^{ij} \left( \partial_i \partial_j F^\alpha - \sum_{l=1}^d \Gamma(m)_{ij}^l \partial_l F^\alpha
	+ \sum_{\beta,\gamma=1}^{n} \Gamma(G)_{\beta \gamma}^\alpha \circ F \ \partial_i F^\beta \partial_j F^\gamma \right) \partial_\alpha
	\\
	 &= \sum_{\alpha=1}^{n} \sum_{i,j=1}^d m^{ij} \left( \partial_i \partial_j F^\alpha - \sum_{l=1}^d \Gamma(m)_{ij}^l \partial_l F^\alpha
	+ \sum_{\beta,\gamma=1}^{n} \Gamma(g)_{\beta \gamma}^\alpha \circ F \ \partial_i F^\beta \partial_j F^\gamma \right) \partial_\alpha
	\\
	&= \Delta_{m,g} f.
\end{align*}
\end{proof}
By the above result, it is clear that a solution to the harmonic map heat flow (\ref{harmonic_map_heat_flow}) also solves 
the extended equation $f_t = \Delta_{m,G} f$. However, the most important application of the above theorem is the following result,
which gives a method to establish existence of solutions to (\ref{harmonic_map_heat_flow}) by solving the extended equation.

\begin{thm}
\label{Theorem_extended_HMF}
Let $(\Gamma, m)$ and $(\M, g)$ be two smooth closed Riemannian manifolds. Suppose that $(\M, g)$ is a totally geodesic submanifold of $(\mathcal{N}, G)$,
and that $i: \mathcal{N} \rightarrow \mathcal{N}$ is an isometry whose fixed point set is given by $\M$. 
Furthermore, let $f_0: \Gamma \rightarrow \mathcal{M}$ be a smooth map. 
Then the solution to 
\begin{equation}
	f_t = \Delta_{m,G} f \quad \textnormal{with $f(\cdot, 0) = f_0(\cdot)$ on $\Gamma$.} 
	\label{extended_HMF}
\end{equation}
solves the harmonic map heat flow (\ref{harmonic_map_heat_flow}). 
\end{thm}
\begin{proof}
The proof follows the argumentation in \cite{Ham75}.
According to Theorem \ref{Theorem_extended_map_Laplacian} we only have to show that $f$ maps onto the submanifold $\M$ as long as $f$ exists.
Since $\M$ is the fixed point set of $i$, we have $i \circ f_0 = f_0$. 
Furthermore, the map $i \circ f$ solves equation (\ref{extended_HMF}) for initial data $i \circ f(\cdot,0) = i \circ f_0 (\cdot) = f_0 (\cdot)$. This can be seen as follows
\begin{align*}
	(i \circ f)_t = (di \circ f) (f_t) = (di \circ f) (\Delta_{m,G} f) = \Delta_{m, i_\ast G} (i \circ f) = \Delta_{m,G} (i \circ f),
\end{align*}
where we have made use of Proposition \ref{Prop_invariance_map_Laplacian} and the fact that $i$ is an isometry.
From the uniqueness of solutions to (\ref{extended_HMF}),
it then follows that $i \circ f = f$. Hence, $f$ must map onto the fixed point set of $i$, which is $\M$.
\end{proof}
\begin{remark}
If $\M$ is embedded into $\mathbb{R}^N$, the manifold $\N$ can be chosen to be a tubular neighbourhood of $\M$.
It is clear that for short times the solution to (\ref{extended_HMF}) remains in this neighbourhood. If we choose the maximal time interval for which 
the solution stays in the neighbourhood and apply the above theorem, we immediately see by a contradiction argument that this maximal time interval must be 
given by the maximal time of existence.     
\end{remark}
Existence of solutions to (\ref{harmonic_map_heat_flow}) can be proved by embedding the Riemannian manifold into some Euclidean space
and by solving the extended equation (\ref{extended_HMF}), where the isometry $i$ and the extended Riemannian metric $G$
are constructed as in Section \ref{section_totally_geodesic_embeddings}. The advantage of this approach is that the extended problem is
formulated for mappings with values in some Euclidean space. This means that there are no constraints which have to be satisfied a priori by the solution.
Instead, the fact that the solution maps into the target manifold is now a property of the solution which is due to a certain invariance of the elliptic operator.
Clearly, this observation is also interesting for numerics. 
Therefore, we here aim to compute an approximation to the extended harmonic map heat flow 
instead of solving the original problem (\ref{harmonic_map_heat_flow}). 
\begin{problem}
Let $(\Gamma, m)$ and $(\M, g)$ be two smooth closed Riemannian manifolds. Suppose that $(\M, g)$ is a totally geodesic submanifold of $(\mathcal{N}, G)$,
and that $i: \mathcal{N} \rightarrow \mathcal{N}$ is an isometry whose fixed point set is given by $\M$. Furthermore, let $f_0: \Gamma \rightarrow \mathcal{M}$ be a smooth map. 
Then find an approximation to the extended harmonic map heat flow (\ref{extended_HMF}).
\end{problem}

\subsubsection*{Example: The extended harmonic map heat flow for a spherical target mani\-fold}
In order to get a better feeling for the extended harmonic map heat flow, we now consider the case when the target manifold $\M$ is the standard $n$-sphere in 
$\mathbb{R}^{n+1}$ with metric induced by the Euclidean metric of the ambient space. 
As we have seen above, the target manifold is then a totally geodesic submanifold of $\mathbb{R}^{n+1} \setminus \{ 0\}$ with Riemannian metric $G(x) := \rho(x) \unit$
as in (\ref{extended_sphere_metric}).
In particular, the sphere inversion is an isometry of the metric $G$. The assumptions in Theorem \ref{Theorem_extended_HMF} 
are therefore satisfied. Next, we derive equation (\ref{extended_HMF}) explicitly for the standard $n$-sphere.
We first observe that $D_\beta G_{\kappa \iota} = D_\beta \rho (x) \delta_{\kappa \iota}$ and $D_\beta \rho(x) = - \frac{2 x_\beta}{|x|^6}$. 
Hence, we obtain for all $\gamma = 1, \ldots, n+1$ that
\begin{equation*}
 f^\gamma_t = \Delta_m f^\gamma + \frac{1}{\tfrac{1}{2} + \tfrac{1}{2|f|^4}} 
 \sum_{\beta=1}^{n+1} \left(- \frac{2 f^\beta}{|f|^6} m(grad_m f^\beta, grad_m f^\gamma)  + \frac{f^\gamma}{|f|^6} m(grad_m f^\beta, grad_m f^\beta)\right),
\end{equation*}
where the Greek indices refer to the Euclidean coordinates of the ambient space. 
If we choose $(\Gamma, m)$ to be the standard $n$-sphere in $\mathbb{R}^{n+1}$ too, we see that the map $f: \Gamma \rightarrow \mathbb{R}^{n+1} \setminus \{ 0 \}$
with $f(x,t):= r(t) x$ and $r(0) > 0$ is an extended harmonic map heat flow if $r(t)$ satisfies the ODE 
$$
	r'(t) = nr(t) \left( \frac{1 - r(t)^4}{1 + r(t)^4} \right),
$$
from which we immediately conclude that the standard $n$-sphere is an attractor. We have therefore found a good candidate for an extension of the harmonic map heat flow,
which we expect to be stable under small perturbations in the following sense: A map $f$ with initial values close to the target manifold has values
close to target manifold as long as it exists.
That this is indeed true will be shown in the next section.  

\subsection{Stability of the extended harmonic map heat flow}
\label{section_stability_HMF}
Let $(\Gamma,m)$ and $(\M,g)$ be as in Theorem $\ref{Theorem_extended_HMF}$ and $f: \Gamma \times [0,T) \rightarrow (\N,G)$ be a solution
to the extended harmonic map heat flow $f_t = \Delta_{m,G} f$. In this section we do not assume that the initial map $f(\cdot,0)$ maps into $\M$.
Furthermore, let $\sigma: \N \rightarrow \mathbb{R}$ be an arbitrary smooth map on $\N$. We define
$$
	\eta := \sigma \circ f.
$$
For the time derivative we obtain
$$
	\eta_t = f^\alpha_t D_\alpha \sigma \circ f, 
$$
and the Laplacian is given by
\begin{align*}
	\Delta_m \eta &= m^{ij} ( \partial_i \partial_j \eta  - \Gamma(m)^k_{ij} \partial_k \eta )
\\	&= m^{ij} ( \partial_i (D_\alpha \sigma \circ f \, \partial_j f^\alpha) - \Gamma(m)^k_{ij} D_\alpha \sigma \circ f \, \partial_k f^\alpha )
\\  &= m^{ij} ( D_\beta D_\alpha \sigma \circ f \, \partial_i f^\beta \partial_j f^\alpha + D_\alpha \sigma \circ f \, \partial_i \partial_j f^\alpha
				- \Gamma(m)^k_{ij} D_\alpha \sigma \circ f \, \partial_k f^\alpha )
\\  &= m^{ij} \partial_i f^\beta \partial_j f^\alpha (D_\beta D_\alpha \sigma \circ f - \Gamma(G)_{\beta\alpha}^\gamma \circ f \, D_\gamma \sigma \circ f)
		+ \Delta_{m,G} f^\alpha D_\alpha \sigma \circ f  				
\\  &= m ( grad_m f^\beta,  grad_m f^\alpha) (\nabla^G_\alpha \nabla^G_\beta \sigma) \circ f
		+ \Delta_{m,G} f^\alpha D_\alpha \sigma \circ f,  			
\end{align*}
where $\nabla^G_\alpha \nabla^G_\beta \sigma$ denotes the Hessian of $\sigma$ with respect to the metric $G$. We conclude that $\eta$ satisfies the following 
reaction-diffusion equation
$$
	\eta_t - \Delta_m \eta = - m ( grad_m f^\beta,  grad_m f^\alpha) (\nabla^G_\alpha \nabla^G_\beta \sigma) \circ f.
$$
Next, we apply this result to the case that the target manifold is the standard $n$-sphere.

\subsubsection*{Example: Stability of the extended harmonic map heat flow for a spherical target manifold}	
\begin{lem}
Let $(\Gamma,m)$ be a smooth closed Riemannian manifold and $(\M,g)$ the standard $n$-sphere in $\mathbb{R}^{n+1}$ with extended metric $G(x) = \rho(x) \unit$ 
on $\mathbb{R}^{n+1} \setminus \{0 \}$.
Furthermore, let $f: \Gamma \times [0,T) \rightarrow \mathbb{R}^{n+1}\setminus \{ 0 \}$ be a solution to the extended harmonic map heat flow $f_t = \Delta_{m,G} f$.
Then, if $||f(\cdot, 0)| - 1| \leq \delta$ on $\Gamma$ for some $\delta \in [0,1)$, we have $||f| - 1| \leq \delta$ on $\Gamma \times [0,T)$. 
\end{lem}
\begin{proof}
We first consider the map $\sigma: \mathbb{R}^{n+1} \setminus \{0\} \rightarrow \mathbb{R}$ with $\sigma(x) = |x| - 1$. Then the first and second derivatives are
$D_\alpha \sigma(x) = \frac{x_\alpha}{|x|}$, and $D_\alpha D_\beta \sigma(x) = \frac{1}{|x|} \left( \delta_{\alpha\beta} - \frac{x_\alpha}{|x|} \frac{x_\beta}{|x|} \right)$.
For the Christoffel symbols of the metric (\ref{extended_sphere_metric}) we obtain
$\Gamma(G)^\gamma_{\alpha\beta}(x) = \tfrac{1}{2 \rho(x)} ( \delta_{\gamma \beta} D_\alpha \rho(x) + \delta_{\gamma \alpha} D_\beta \rho(x) - \delta_{\alpha \beta} D_\gamma \rho(x) )$
with $D_\alpha \rho(x) = - \frac{2 x_\alpha}{|x|^6}$. Hence, the Hessian of $\sigma$ is given by
\begin{align*}
\nabla^G_\alpha \nabla^G_\beta \sigma (x) &= D_\alpha D_\beta \sigma(x) - \Gamma(G)^\gamma_{\alpha \beta}(x) D_\gamma \sigma(x) 
\\
&= \tfrac{1}{|x|} \left( \delta_{\alpha\beta} - \tfrac{x_\alpha}{|x|} \tfrac{x_\beta}{|x|} \right) 
	- \tfrac{1}{2 \rho(x) |x|} (x_\beta D_\alpha \rho(x) + x_\alpha D_\beta \rho(x) - \delta_{\alpha \beta} x_\gamma D_\gamma \rho(x) )
\\
&=  \tfrac{1}{|x|} \left( \delta_{\alpha\beta} - \tfrac{x_\alpha}{|x|} \tfrac{x_\beta}{|x|} \right)  
	+ \tfrac{1}{\rho(x)|x|^5} \left( 2 \tfrac{x_\alpha}{|x|} \tfrac{x_\beta}{|x|} - \delta_{\alpha \beta} \right)
\\
&= \tfrac{1}{|x|} \left( \delta_{\alpha\beta} \tfrac{|x|^4 - 1}{|x|^4 + 1} + \tfrac{3 - |x|^4}{1 + |x|^4} \tfrac{x_\alpha}{|x|} \tfrac{x_\beta}{|x|} \right)
\\
&= \tfrac{1}{|x|(1 + |x|^4)} \left( \delta_{\alpha\beta} (1 + |x|)(1 + |x|^2) \sigma(x) + (3 - |x|^4) D_\alpha \sigma(x) D_\beta \sigma(x) \right).
\end{align*}
The map $\eta: \Gamma \rightarrow \mathbb{R}$ with $\eta := \sigma \circ f =  |f| - 1$ therefore satisfies
\begin{equation}
	\eta_t - \Delta_m \eta + \tfrac{3 - |f|^4}{|f|(1 + |f|^4)} m ( grad_m \eta,  grad_m \eta) = - \eta \, 
	\underbrace{\tfrac{(1 + |f|)(1 + |f|^2)}{|f|(1 + |f|^4)} m ( grad_m f^\alpha,  grad_m f^\alpha)}_{ \geq 0 }.
	\label{heat_equation_distance_function}
\end{equation}
Now, let $\varepsilon \geq 0$ be the maximal time for which $f$ maps into $\mathbb{R}^{n+1} \setminus B_{(1-\delta)/2}(0)$. Since $f$ is continuous, $\varepsilon$ has to be positive.
From the maximum principle, see, for example, Theorem 3.1.1 in \cite{Top06}, it follows that $|\eta| \leq \delta$ on $[0, \varepsilon)$. From the maximality of $\varepsilon$ we finally conclude that $\varepsilon = T$.
\end{proof}

\subsubsection*{Stability of the extended flow for isometrically embedded target manifolds}
The above result can be generalized to isometrically embedded target manifolds of co-dimension one in Euclidean spaces.
The crucial point in the proof of the following lemma is to use the signed distance function, 
which has already been used in the definition of the metric $G$ in (\ref{extended_metric_hypersurfaces}). 
\begin{lem}
Let $(\Gamma,m)$ be a smooth closed Riemannian manifold and $\M \subset \mathbb{R}^{n+1}$ an orientable, smooth closed hypersurface in $\mathbb{R}^{n+1}$.
Let $\N_T$ be a tubular neighbourhood of $\M$ of fixed width such that the decomposition $x = a(x) + d(x) Dd(x)$ with $a(x) \in \M$ is unique on $\N_T$.
Furthermore, let $G$ be the extended metric defined in (\ref{extended_metric_hypersurfaces}). 
If $f: \Gamma \times [0,T) \rightarrow \N_T$ is a solution to the flow $f_t = \Delta_{m,G} f$ 
satisfying $|d(f(\cdot,0))| \leq \delta$ on $\Gamma$, then $|d(f)| \leq \delta$ on $\Gamma \times [0,T)$.
\end{lem}
\begin{proof}
Choosing $\sigma = d$, we obtain that
$$
	m(grad_m f^\alpha, grad_m f^\beta) (\nabla^G_\alpha \nabla^G_\beta \sigma ) \circ f = m(grad_m f^\alpha, grad_m f^\beta) \,
		d \circ f \left( D_\alpha D_\gamma d D_\beta D_\gamma d \right) \circ f,
$$
see Lemma \ref{lemma_distance_function} in the Appendix for details. Hence, $\eta := d \circ f$ satisfies the equation 
$$
	\eta_t - \Delta_m \eta = - \eta \, \underbrace{m(grad_m f^\alpha, grad_m f^\beta) \left( D_\alpha D_\gamma d D_\beta D_\gamma d \right) \circ f}_{ \geq 0}. 
$$
The claim then follows again from the maximum principle.
\end{proof}
We leave the study of the stability of the extended flow for (not necessarily isometrically) embedded target manifolds of higher co-dimensions for future research.
The starting point for the general case could be to consider generalized (non-Euclidean) distance functions.

\subsection{Weak formulation of the extended harmonic map heat flow}
\label{section_weak_formulation_HMF}
In this subsection we derive a weak formulation of the extended flow (\ref{extended_HMF}) which is suitable for a finite element discretization.
\begin{lem}
\label{lemma_weak_formulation}
Let $(\Gamma,m)$ and $(\mathcal{N},G)$ be two smooth Riemannian manifolds, where $\Gamma$ is closed and $\mathcal{N} \subset \mathbb{R}^N$ is some open subset of $\mathbb{R}^N$. 
Furthermore, let $f: (\Gamma, m) \times [0,T) \rightarrow (\mathcal{N}, G)$ be a solution to (\ref{extended_HMF}), then
\begin{align*}
&\int_\Gamma G(f)(f_t, \psi) \, do + \int_\Gamma G_{\alpha \beta}(f) \, m(grad_m f^\alpha, grad_m \psi^\beta) \, do \\
&=  - \tfrac{1}{2} \int_\Gamma (D_\beta G_{\kappa \iota} \circ f) \,  m(grad_m f^\kappa, grad_m f^\iota) \psi^\beta \, do
\end{align*}
for all vector fields $\psi \in H^{1,2}(\Gamma, \mathbb{R}^N)$ and all $t \in (0,T)$.
\end{lem}
\begin{proof}
We test equation (\ref{extended_HMF}) with $\psi$ using the Riemannian metric $G$ on $\mathcal{N}$ and integrate on $\Gamma$ with respect to the metric $m$. This leads to
$$
\int_\Gamma G(f)(f_t, \psi) \, do = \int_\Gamma G(f)(\Delta_{m,G} f, \psi) \, do. 
$$
Let $\{\xi_l \}_{l}$ be a partition of unity subordinate to an atlas for $\Gamma$. We then consider the terms 
\begin{align*}
& \int_\Gamma G(f)(\Delta_{m,G} f, \psi) \xi_l \, do 
\\
&= \int_\Omega G_{\alpha \beta}(F) m^{ij} ( \partial_i \partial_j F^\alpha - \Gamma(m)^k_{ij} \partial_k F^\alpha
	+ \Gamma(G)^\alpha_{\kappa \iota} \circ F \, \partial_i F^\kappa \partial_j F^\iota ) \Psi^\beta \Xi_l \sqrt{m} \, d\theta
\\
&=	- \int_\Omega m^{ij} \partial_j F^\alpha \partial_i (G_{\alpha \beta}(F) \Psi^\beta \Xi_l ) \sqrt{m} \, d\theta
	+ \int_\Omega G_{\alpha \beta}(F) m^{ij} \Gamma(G)^\alpha_{\kappa \iota} \circ F \, \partial_i F^\kappa \partial_j F^\iota \Psi^\beta \Xi_l \sqrt{m} \, d\theta
\\
&=	- \int_\Omega G_{\alpha \beta}(F) m^{ij} \partial_j F^\alpha \partial_i ( \Psi^\beta \Xi_l ) \sqrt{m} \, d\theta
	- \int_\Omega m^{ij} \partial_j F^\alpha D_\kappa G_{\alpha\beta}\circ F \, \partial_i F^\kappa \Psi^\beta \Xi_l \sqrt{m} \, d\theta
	\\
&\quad	+ \int_\Omega G_{\alpha \beta}(F) m^{ij} \Gamma(G)^\alpha_{\kappa \iota} \circ F \, \partial_i F^\kappa \partial_j F^\iota \Psi^\beta \Xi_l \sqrt{m} \, d\theta	
\end{align*}
\begin{align*}
&=	- \int_\Omega G_{\alpha \beta}(F) m^{ij} \partial_j F^\alpha \partial_i ( \Psi^\beta \Xi_l ) \sqrt{m} \, d\theta
	- \int_\Omega m^{ij} D_\kappa G_{\iota\beta}\circ F \, \partial_i F^\kappa \partial_j F^\iota \Psi^\beta \Xi_l \sqrt{m} \, d\theta
	\\
&\quad	+ \tfrac{1}{2} 
	\int_\Omega m^{ij} ( D_\kappa G_{\beta \iota} + D_{\iota} G_{\kappa \beta} - D_\beta G_{\kappa \iota}) \circ F \, \partial_i F^\kappa \partial_j F^\iota \Psi^\beta \Xi_l \sqrt{m} \, d\theta	
\\
&= 	- \int_\Omega G_{\alpha \beta}(F) m^{ij} \partial_j F^\alpha \partial_i ( \Psi^\beta \Xi_l ) \sqrt{m} \, d\theta
	- \tfrac{1}{2} \int_\Omega m^{ij} D_\beta G_{\kappa \iota} \circ F \, \partial_i F^\kappa \partial_j F^\iota \Psi^\beta \Xi_l \sqrt{m} \, d\theta,	
\end{align*}
where $F: \Omega \rightarrow \mathbb{R}^N$ and $\Xi_l: \Omega \rightarrow \mathbb{R}$ are the local coordinate functions on $\Omega$ of the map $f: \Gamma \rightarrow \mathcal{N}$
and the function $\xi_l: \Gamma \rightarrow \mathbb{R}$. We rewrite the above result using the tangential gradient with respect to $m$ on $\Gamma$, that is
\begin{align*}
\int_\Gamma G(f) (\Delta_{m,G} f, \psi) \xi_l \, do 
&= - \int_\Gamma G_{\alpha \beta}(f) \, m(grad_m f^\alpha, grad_m (\psi^\beta \xi_l)) \, do \\
&\quad  - \tfrac{1}{2} \int_\Gamma (D_\beta G_{\kappa \iota} \circ f) \,  m(grad_m f^\kappa, grad_m f^\iota) \psi^\beta \xi_l \, do.
\end{align*}
Summing up all terms from the partition of unity, we finaly obtain
\begin{align*}
&\int_\Gamma G(f)(f_t, \psi) \, do + \int_\Gamma G_{\alpha \beta}(f) \, m(grad_m f^\alpha, grad_m \psi^\beta) \, do \\
&=  - \tfrac{1}{2} \int_\Gamma (D_\beta G_{\kappa \iota} \circ f) \,  m(grad_m f^\kappa, grad_m f^\iota) \psi^\beta \, do.
\end{align*}
\end{proof}

\subsubsection*{Example: The weak formulation for the spherical target manifold}
We return to the case when $\M$ is the standard $n$-sphere in $\mathbb{R}^{n+1}$ with extended metric $G(x) = \rho(x) \unit$ defined in (\ref{extended_sphere_metric}).
Since $D_\beta G_{\kappa \iota} = D_\beta \rho (x) \delta_{\kappa \iota}$, the weak formulation of the extended harmonic map heat flow 
on a $d$-dimensional closed hypersurface $\Gamma \subset \mathbb{R}^{d+1}$ with metric induced from the ambient space then is
$$
\int_\Gamma f_t \cdot \psi \,\rho(f) \, do + \int_\Gamma \grad f : \grad \psi \, \rho(f) \, do = - \tfrac{1}{2} \int_\Gamma  D\rho (f) \cdot \psi \, |\grad f|^2 \, do,
\quad \forall \psi \in H^1(\Gamma, \mathbb{R}^{n+1}).
$$
Inserting $D\rho(x) = -2 \tfrac{x}{|x|^6}$, we obtain 
$$
\int_\Gamma f_t \cdot \psi \, \left( \tfrac{1}{2} + \tfrac{1}{2|f|^4} \right) \, do
+ \int_\Gamma \grad f : \grad \psi \, \left( \tfrac{1}{2} + \tfrac{1}{2|f|^4} \right) \, do = \int_\Gamma  f \cdot \psi \, \tfrac{|\grad f|^2}{|f|^6} \, do,
$$
for all $\psi \in H^1(\Gamma, \mathbb{R}^{n+1})$.

\section{Discretization}
\subsection{The surface finite element method}
\label{section_SFEM}
Throughout this section let $\Gamma \subset \mathbb{R}^{d+1}$ be an orientable $d$-dimensional closed hypersurface of class $C^{3}$, for $d \leq 3$.
$\Gamma$ is supposed to be approximated by a polyhedral hypersurface 
$$
	\Gamma_h = \bigcup_{T \in \mathcal{T}} T,
$$
which is the union of $d$-dimensional, non-degenerate simplices $T$ in $\mathcal{T}$, whose vertices sit on $\Gamma$.
The triangulation $\mathcal{T}$ is supposed to be admissible, which means that either two simplices of the triangulation have empty cross section or
the cross section is a sub-simplex of both. 
The maximal diameter of the simplices $T \in \mathcal{T}$ is denoted by $h$. We assume that the triangulation is such that the inner radii of the simplices $T$ are $\geq C_1 h$, 
where $C_1 > 0$ is some constant.
The tangential projection $P_h = \unit - \nu_h \otimes \nu_h$ and the tangential gradient $\gradh$ on $\Gamma_h$
can be defined piecewise on each simplex $T \in \mathcal{T}$. Note, however, that the unit normal $\nu_h$ is in general only well-defined up to multiplication by $-1$. 
We now assume that $\Gamma_h$ is contained in a tubular neighbourhood of $\Gamma$
with the property that for each point $x$ there is a unique point $a(x) \in \Gamma$ such that $x = a(x) + d(x) Dd(x)$ holds.
The restriction $a_{\Gamma_h}: \Gamma_h \rightarrow \Gamma$ of the projection $a$ to $\Gamma_h$ 
is supposed to be bijective. For a function $f$ on $\Gamma_h$ we define the lift $f^l$ onto $\Gamma$
by $f^l := f \circ a_{\Gamma_h}^{-1}$. The lift of a function $f$ on $\Gamma$ onto $\Gamma_h$ is denoted by $f_l := f \circ a_{\Gamma_h}$.
$\nu(x) := Dd(x)$ is the extension of the outward unit normal to a neighbourhood of $\Gamma$, which is constant in the normal direction.
$\mathcal{H}(x) := D^2 d(x)$ is the corresponding extension of the shape operator $\mathcal{H} = \grad \nu$ on $\Gamma$.
We set $P(x) := \unit - \nu(x) \otimes \nu(x)$ to be the extension of the tangential projection on $\Gamma$.
In this section a constant $C(\Gamma,f)$ depending on quantities such as $\Gamma$ and $f$ might change from line to line.

The following statement gives detailed information about the quality of the approximation of $\Gamma$ by $\Gamma_h$.
\begin{prop}[Geometric estimates]
\label{prop_geometric_estimates}
Let $d$ denote the oriented distance function to $\Gamma$, $\mu_h$ the ratio between the volume forms on $\Gamma_h$ and $\Gamma$,
and let $R_h := P (\unit - d \mathcal{H}) P_h (\unit - d \mathcal{H}) P$. Then the following estimates hold 
\begin{align*}
	& \| d \|_{L^\infty(\Gamma_h)} \leq C(\Gamma) h^2
\\
	& \| 1 - \mu_h \|_{L^\infty(\Gamma_h)} \leq C(\Gamma) h^2
\\
	& \| P - R_h \|_{L^\infty(\Gamma_h)} \leq C(\Gamma) h^2
\end{align*}
\end{prop}
\begin{proof}
See \cite{Dz88, He04}. Note that our definition of $\mu_h$ is the inverse of the corresponding quantity in those papers.
\end{proof}
A direct consequence of this result is the following equivalence of norms.
\begin{prop}[Equivalence of norms]
\label{prop_equivalence_of_norms}
Let $1 \leq p \leq \infty$. Then we have
\begin{align*}
 & \tfrac{1}{C(\Gamma)} \| f \|_{L^p(\Gamma_h)} \leq \| f^l \|_{L^p(\Gamma)} \leq C(\Gamma) \| f \|_{L^p(\Gamma_h)}
 \quad \textnormal{if $f \in L^p(\Gamma_h)$,}
 \\
 & \tfrac{1}{C(\Gamma)} \| \gradh f \|_{L^p(\Gamma_h)} \leq \| \grad f^l \|_{L^p(\Gamma)}  \leq C(\Gamma) \| \gradh f \|_{L^p(\Gamma_h)}
 \quad \textnormal{if $f \in W^{1,p}(T)$ for all $T \in \mathcal{T}$,}
 \\
 & \| (\gradh)^k f \|_{L^p(\Gamma_h)} \leq C(\Gamma) \sum_{j=1}^{k} \| (\grad)^j f^l \|_{L^p(\Gamma)}
 \quad \textnormal{if $f \in W^{k,p}(T)$ for all $T \in \mathcal{T}$.}
\end{align*}
\end{prop}
\begin{proof}
See \cite{Dem09,Dz88}.
\end{proof}
The finite element space $V_h$ is defined as $V_h := \{ \psi \in C^0(\Gamma_h): \textnormal{$\psi_{T}$ is affine for all $T \in \mathcal{T}$} \}$.
On $\Gamma_h$ the Lagrange interpolation operator $I_h: C^0(\Gamma_h) \rightarrow V_h$ is defined like in the flat space
and corresponding interpolation estimates hold.
We define the lifted interpolation operator $I^l_h: C^0(\Gamma) \rightarrow V_h^l := \{ \psi_h^l: \psi_h \in V_h\}$ by
$I_h^l f := (I_h f_l)^l$ and find the following interpolation estimates
\begin{align}
 & \| f - I_h^l f \|_{L^p} \leq C(\Gamma) h^2 \| f \|_{W^{2,p}}
 \label{interpolation_estimate_1}
 \\
 & \| \grad (f - I_h^l f) \|_{L^p} \leq C(\Gamma) h \| f \|_{W^{2,p}}
 \label{interpolation_estimate_2}
\end{align}
for all $f \in W^{2,p}(\Gamma)$, $2 \leq p \leq \infty$, see \cite{Dem09}.
\subsection{The discrete problems}
\label{section_discrete_problems}
We divide the maximal interval $[0,T)$ of existence into time steps of length $\tau > 0$ and write $f^m(\cdot) := f(\cdot, m\tau)$.
We use the time discretization for a linearization of the problem and propose the following scheme for the computation of
the harmonic map heat flow into the $n$-sphere $\mathbb{S}^n \subset \mathbb{R}^{n+1}$.
\begin{alg}
\label{algorithm_sphere}
Let $\Gamma_h \subset \mathbb{R}^{d+1}$ be a $d$-dimensional polyhedral hypersurface and $f^0_h \in (V_h)^{n+1}$ with $||f^0_h| - 1| < \delta$ on $\Gamma_h$
for some $\delta \in (0,1)$.
For all $m \in \mathbb{N}_0$ with $(m + 1)\tau < T$ find $f_h^{m+1} \in (V_h)^{n+1}$ such that 
\begin{align*}
&\tfrac{1}{\tau} \int_{\Gamma_h} (f^{m+1}_h - f^m_h) \cdot \psi_h \left( \tfrac{1}{2} + \tfrac{1}{2|f^m_h|^4} \right) \, do
+ \int_{\Gamma_h} \gradh f_h^{m+1} : \gradh \psi_h \left( \tfrac{1}{2} + \tfrac{1}{2|f^m_h|^4} \right) \, do
\\
&= \int_{\Gamma_h} f_h^m \cdot \psi_h \tfrac{|\gradh f_h^m|^2}{|f_h^m|^6} \, do \qquad \textnormal{for all $\psi_h \in (V_h)^{n+1}$}.
\end{align*}
\end{alg}

\begin{remark}
The above algorithm would obviously break down if $|f^m_h(p)| = 0$ in some point $p \in \Gamma_h$. However, due to the stability proved in Section \ref{section_stability_HMF}
and the initial condition $||f^0_h| - 1| < \delta$, this should not happen. Otherwise, it would be allowed to modify the metric $G$ defined in (\ref{extended_sphere_metric}) 
outside a neighbourhood of the target manifold, 
for example, by smoothly modifying the function $\rho(x)$ for $|x| < 1/2$. But since we never observed that $|f_h^m|$ became critically small in our experiments, 
a change of the metric $G$ does not seem to be necessary from a practical point of view.
\end{remark}

\begin{remark}
\label{remark_stationary_problem}
An important application of Algorithm \ref{algorithm_sphere} is the following problem. Find $f_h \in (V_h)^{n+1}$ 
with $||f_h| - 1| < \delta$ on $\Gamma_h$ for some $\delta \in (0,1)$ such that
\begin{align}
\label{stationary_problem}
\int_{\Gamma_h} \gradh f_h : \gradh \psi_h \left( \tfrac{1}{2} + \tfrac{1}{2|f_h|^4} \right) \, do
= \int_{\Gamma_h} f_h \cdot \psi_h \tfrac{|\gradh f_h|^2}{|f_h|^6} \, do \qquad \textnormal{for all $\psi_h \in (V_h)^{n+1}$}.
\end{align}
In order to compute an approximation to this problem, we solve the discrete flow in Algorithm \ref{algorithm_sphere} as long as
$$
	\max_{p_j \in V} \tfrac{1}{\tau} |f^{m+1}_h(p_j) - f_h^m(p_j)| > tol,
$$ 
where $V$ is the set of the mesh vertices of $\Gamma_h$ and $tol > 0$ is some chosen threshold.
\end{remark}

\begin{remark}
The idea for the general case is based on the discretization of the weak formulation in Lemma \ref{lemma_weak_formulation}.
Assume we have found an appropriate way to discretize $G \circ f$ and the derivative $DG \circ f$, then solving the equation 
\begin{align*}
&\tfrac{1}{\tau} \int_{\Gamma_h} G_{h \alpha \beta}(f_h^m)(f^{\alpha, m+1}_h - f^{\alpha, m}_h,\psi^\beta_h) \, do
+ \int_{\Gamma_h} G_{h\alpha \beta}(f_h^m) \gradh f_h^{\alpha, m+1} \cdot \gradh \psi_h^\beta \, do
\\
&= - \tfrac{1}{2} \int_{\Gamma_h} (D_\beta G_{\kappa \iota})_h (f_h^m) \gradh f_h^{\kappa,m} \cdot \gradh f_h^{\iota,m} \psi^\beta_h \, do \qquad \textnormal{for all $\psi_h \in (V_h)^{n+1}$}.
\end{align*}
would lead to a numerical scheme for the computation of the harmonic map heat flow for non-spherical target manifolds $\M \subset \mathbb{R}^{n+1}$. 
Possible choices are $G_h(f_h^m) := G(f_h^m)$ or $G_h(f_h^m) := I_h (G(f_h^m))$ and respectively,
$(DG)_h(f^m_h) := DG(f^m_h)$ or $(DG)_h(f^m_h) := I_h (DG(f^m_h))$.
\end{remark}
Since the numerical analysis of the general case is beyond the scope of this paper, we will henceforward restrict to the case of spherical target manifolds.

\subsection{Numerical analysis}
\label{section_numerical_analysis}
The harmonic map heat flow is often considered in order to obtain a harmonic map. The basic idea, which we also followed in this paper, see Remark \ref{remark_stationary_problem}, 
is to study the limit of the flow as $t$ tends to infinity.
For a numerical scheme based on a discretization of the flow, this means that two different kinds of limits are involved -- one with respect to the discretization,
that is with respect to the mesh size $h$ and the time step size $\tau$, and the other with respect to the evolving time $t$. Therefore, two different numerical analysis problems
arise depending on which limit is considered first. The study of $h,\tau \searrow 0$ for a fixed time interval $[0,t]$ leads to the numerical analysis of the flow and 
relevant error estimates would typically depend on the time interval $[0,t]$. It is usually unclear how to control these estimates as $t$ tends to infinity.
On the other hand, letting $t \rightarrow \infty$ first (whilst keeping $h$ fixed) and $h \searrow 0$ afterwards leads to the numerical analysis of the stationary problem.
Although, in a numerical experiment one actually never really reaches the limit of such a process, the experiment is usually better described by the latter scenario,
since in a typical experiment the mesh size $h$ and time step size $\tau$ is fixed (apart from changes due to adaptivity), while the computation is run as long as the velocity
is above a certain threshold. For this reason we will here study the numerical analysis of the stationary problem 
and leave the numerical analysis of the flow open for further research. For the stationary problem we face two different questions.
First, does a sequence of discrete solutions converge to a harmonic map as $h \searrow 0$ and secondly, is each harmonic map approximated by a sequence of discrete solutions
(recovery problem)? In this paper only the latter problem will be discussed. 

The numerical analysis in the following subsection is similar to the work in \cite{St14}, 
where the author proved the existence of a recovery sequence for a numerical scheme of the harmonic map equation.
The main difference is that in \cite{St14} the constraints of the harmonic map equation are formulated with the help of Lagrange multipliers which leads to a saddle point problem, 
whereas we consider an unconstrained problem. We therefore do not have to check a Babu\v{s}ka-Brezzi condition, which was a significant part of the proof in \cite{St14}.
However, we will also make use of a quantitative version of the Inverse Function Theorem, which in \cite{DH99} was used for the numerical analysis of the discrete Plateau problem.
The proof of Lemma \ref{b_decomposition} is inspired by the proof of Theorem 3.2 in \cite{HTW09}, where a saddle point problem similar to that in \cite{St14} is studied. 
  
\subsubsection{The stationary case}
The aim of this section is to prove that for every harmonic map $f: \Gamma \rightarrow \mathbb{S}^1 \subset \mathbb{R}^2$ 
there exists a recovery sequence of discrete solutions to (\ref{stationary_problem}) that converges to $f$ as $h \searrow 0$.
We first quote the following quantitative version of the Inverse Function Theorem, see Lemma 5.1 in \cite{DH99}, which lies at the heart of our analysis. 
\begin{lem}
\label{lemma_inverse_function_theorem}
Let $\mathcal{X}$ be an affine Banach space with Banach space $\hat{X}$ as tangent space, and let $Y$ be a Banach space. Suppose $x_0 \in \mathcal{X}$
and $\mathcal{F} \in C^1(\mathcal{X},Y)$. Assume there are positive constants $\alpha, \beta, \delta$ and $\varepsilon$ such that
\begin{align*}
& \| \mathcal{F}(x_0) \|_Y \leq \delta, \\
& \| \mathcal{F}'(x_0)^{-1} \|_{L(Y,\hat{X})} \leq \alpha^{-1}, \\
& \| \mathcal{F}'(x) - \mathcal{F}'(x_0) \|_{L(\hat{X},Y)} \leq \beta \quad \textnormal{for all $x \in \overline{B}_\varepsilon (x_0)$,}
\end{align*}
where $\beta < \alpha$ and $\delta \leq (\alpha - \beta) \varepsilon$. Then there exists a unique $x_\ast \in \overline{B}_\varepsilon(x_0)$ such that $\mathcal{F}(x_\ast) = 0$.
\end{lem}
Now, the idea is to apply this theorem to the first variation of the discrete Dirichlet energy
$$
	E_h(f_h) := \frac{1}{2} \int_{\Gamma_h} |\gradh f_h|^2 \left( \tfrac{1}{2} + \tfrac{1}{2 |f_h|^4} \right) \, do,
$$
where $f_h \in (V_h)^{n+1}$. 
For functions with $f_h(x) = 0$ in some point $x \in \Gamma_h$ there are ways to change the energy $E_h$ in a suitable way without any consequences for the statements below.
We fully ignore this problem here, since such functions do not occur in our analysis. 
We define $E'_h(f_h)(\psi_h) := \frac{d}{d\varepsilon} E_h(f_h + \varepsilon \psi_h) \big|_{\varepsilon = 0} $
and $E''_h(f_h)(\psi_h, \psi_h) := \frac{d^2}{d\varepsilon^2} E_h(f_h + \varepsilon \psi_h) \big|_{\varepsilon = 0}$ for all $\psi_h \in (V_h)^{n+1}$.
A short calculation shows that
\begin{align}
	&E'_h(f_h)(\psi_h)
	= \int_{\Gamma_h} \gradh f_h : \gradh \psi_h \left( \tfrac{1}{2} + \tfrac{1}{2 |f_h|^4} \right) \, do 
 		- \int_{\Gamma_h} | \gradh f_h|^2 \, \tfrac{f_h \cdot \psi_h}{|f_h|^6} \, do,
 		\label{first_variation_discrete_energy}
\\		
	&E''_h(f_h)(\psi_h, \psi_h) = \int_{\Gamma_h} |\gradh \psi_h|^2 \left( \tfrac{1}{2} + \tfrac{1}{2 |f_h|^4} \right) \, do
			- 4 \int_{\Gamma_h} \gradh f_h : \gradh \psi_h \, \tfrac{f_h \cdot \psi_h}{|f_h|^6} \, do
\nonumber			
\\ &\hspace*{2.8cm}
			- \int_{\Gamma_h} | \gradh f_h|^2 \, \tfrac{|\psi_h|^2}{|f_h|^6} \, do + 6 \int_{\Gamma_h} | \gradh f_h |^2 \, \tfrac{[f_h \cdot \psi_h]^2}{|f_h|^8} \, do.
		\label{second_variation_discrete_energy}	
\end{align}
The first and second variations of the Dirichlet energy
$$
	E(f) = \frac{1}{2} \int_\Gamma |\grad f|^2 \left( \tfrac{1}{2} + \tfrac{1}{2|f|^4} \right) \, do
$$
for maps $f: \Gamma \rightarrow \mathbb{R}^{n+1} \setminus \{ 0 \}$ are given by
\begin{align}
	& \frac{d}{d\varepsilon} E(f + \varepsilon \psi) \big|_{\varepsilon = 0} 
	= \int_\Gamma \grad f : \grad \psi \left( \tfrac{1}{2} + \tfrac{1}{2 |f|^4} \right) \, do 
 		- \int_\Gamma | \grad f|^2 \, \tfrac{f \cdot \psi}{|f|^6} \, do,
\\		
	& \frac{d^2}{d\varepsilon^2} E(f + \varepsilon \psi) \big|_{\varepsilon = 0} = \int_\Gamma |\grad \psi|^2 \left( \tfrac{1}{2} + \tfrac{1}{2 |f|^4} \right) \, do
			- 4 \int_\Gamma \grad f : \grad \psi \, \tfrac{f \cdot \psi}{|f|^6} \, do
			- \int_\Gamma | \grad f|^2 \, \tfrac{|\psi|^2}{|f|^6} \, do
\nonumber			
\\ &\hspace*{3.2cm}
			+ 6 \int_\Gamma | \grad f |^2 \, \tfrac{[f \cdot \psi]^2}{|f|^8} \, do.
	\label{second_variation_continuous_case}			
\end{align}
Now, let $n=1$ and $f: \Gamma \rightarrow \mathbb{S}^1 \subset \mathbb{R}^2$ be a harmonic map into the $1$-sphere. Then we have
$$
	\frac{d^2}{d\varepsilon^2} E(f + \varepsilon \psi) \big|_{\varepsilon = 0} = \int_\Gamma |\grad \psi|^2 - 4 \grad f : \grad \psi \, (f \cdot \psi) - | \grad f|^2 \, |\psi|^2 
					+ 6 | \grad f |^2 \, [f \cdot \psi]^2 \, do.
$$
We define the bilinear form $b: (H^1(\Gamma))^2 \times (H^1(\Gamma))^2 \rightarrow \mathbb{R}$ by
\begin{equation}
\label{b_definition}
	b(\psi, \psi) := \int_\Gamma |\grad \psi|^2 - 4 \grad f : \grad \psi \, (f \cdot \psi) - | \grad f|^2 \, |\psi|^2 
					+ 6 | \grad f |^2 \, [f \cdot \psi]^2 \, do
\end{equation}
and polarization. We decompose $\psi$ as $\psi = \psi_\nu f + \psi_\tau f^\perp$ with $(f^\perp_1, f^\perp_2) := (- f_2, f_1)$,
that is $\psi_\nu = \psi \cdot f$ and $\psi_\tau = \psi \cdot f^\perp$. We obtain
$$
	b(\psi, \psi) = \int_\Gamma |\grad \psi_\nu |^2 + |\grad \psi_\tau |^2 \, do + 2 \int_\Gamma |\grad f|^2 \psi_\nu^2 \, do,
$$
see Lemma \ref{b_decomposition} in the Appendix.

If $E(f) = 0$ then $f$ is  clearly constant. In this case, the lift $f_l$ is a solution to the discrete problem and nothing has to be done.
In the following we therefore assume that $E(f) \neq 0$, that is $\int_\Gamma |\grad f|^2 \, do \neq 0$.  
We will now show that the bilinear form $b$ restricted to the subspace 
$$
	X := \left\{ \psi \in (H^1(\Gamma))^2 : \int_\Gamma \psi_\tau \, do = 0 \right\}
$$
of $(H^1(\Gamma))^2$ is continuous and coercive with respect to the usual $\| \cdot \|_{H^1}$-norm on $\Gamma$.
\begin{lem}
\label{lemma_coercivity}
Let $f: \Gamma \rightarrow \mathbb{S}^1 \subset \mathbb{R}^2$ be a smooth harmonic map on the closed hypersurface $\Gamma \subset \mathbb{R}^{d+1}$.
The bilinear form $b: X \times X \rightarrow \mathbb{R}$ satisfies
\begin{align*}
& b(\psi, \psi) \leq C(\Gamma, f) \|\psi \|_{H^1}^2 \quad \forall \psi \in X,
\\
& b(\psi, \psi) \geq \lambda(\Gamma, f) \|\psi \|_{H^1}^2 \quad \forall \psi \in X,
\end{align*}
for some $C(\Gamma, f), \lambda(\Gamma, f) > 0$.
\end{lem}
\begin{proof}
We have
\begin{align*}
	b(\psi, \psi) &= \int_\Gamma |\grad \psi_\nu|^2 + |\grad \psi_\tau|^2 \, do + 2 \int_\Gamma |\grad f|^2 \psi_\nu^2 \, do
	\\
	&\leq \| \grad \psi_\nu \|_{L^2}^2 + \| \grad \psi_\tau \|_{L^2}^2 + C(f) \|\psi_\nu \|_{L^2}^2
	\\
	&\leq C(f) \| \psi_\nu \|_{H^1}^2 + \| \psi_\tau \|_{H^1}^2. 
\end{align*}
Furthermore, it is easy to show that
$$
	\tfrac{1}{C(f)} \| \psi \|_{H^1} \leq \| \psi_\nu \|_{H^1} + \| \psi_\tau \|_{H^1} \leq C(f) \| \psi \|_{H^1}.
$$
This proves the continuity. The Poincar\'e inequality gives
\begin{align*}
	\| \psi \|_{H^1} &\leq C(f) (\| \psi_\nu \|_{H^1} + \| \psi_\tau \|_{H^1} )
	\\
	&\leq C(\Gamma, f) ( \| \overline{\psi}_\nu \|_{L^2}  + \| \grad \psi_\nu \|_{L^2} + \| \grad \psi_\tau \|_{L^2} ),
\end{align*}
where $\overline{\psi}_\nu := \tfrac{1}{|\Gamma|} \int_\Gamma \psi_\nu \, do$ is the mean value of $\psi_\nu$.
Since we assume that $\int_\Gamma |\grad f|^2 \, do \neq 0$, it follows that $\|\overline{\psi}_\nu \|^2_{L^2} = C(\Gamma, f) \int_\Gamma |\grad f|^2 (\overline{\psi}_\nu)^2 \, do$.
It therefore remains to show that there is some constant $C \geq 0$ such that
$$
	\int_\Gamma |\grad f|^2 (\overline{\psi}_\nu)^2 \, do \leq C \left( 2 \int_\Gamma |\grad f|^2 \psi_\nu^2 \, do + \int_\Gamma |\grad \psi_\nu|^2 \, do \right).
$$
If the statement was wrong, there would be a sequence $u_k \in H^1(\Gamma)$ such that
$$
	\int_\Gamma |\grad f|^2 (\overline{u}_k)^2 \, do > k \left( 2 \int_\Gamma |\grad f|^2 u_k^2 \, do + \int_\Gamma |\grad u_k|^2 \, do \right).
$$
Without loss of generality we can assume that the mean value $\overline{u}_k$ of $u_k$ satisfies $\overline{u}_k = 1$ for all $k \in \mathbb{N}$.
Hence,
\begin{equation*}
	2 \int_\Gamma |\grad f|^2 u_k^2 \, do + \int_\Gamma |\grad u_k|^2 \, do \rightarrow 0 \quad \textnormal{for $k \rightarrow \infty$},
\end{equation*} 
and in particular,
\begin{equation*}
	\| u_k - 1 \|_{H^1} = \| u_k - \overline{u}_k \|_{H^1} \leq C(\Gamma) \| \grad u_k \|_{L^2} \rightarrow 0 \quad \textnormal{for $k \rightarrow \infty$}.
\end{equation*} 
That is $u_k \rightarrow 1$ in $H^1(\Gamma)$. This implies that
\begin{align*}
 0 = \lim_{k \rightarrow \infty} \int_\Gamma |\grad f|^2 u_k^2 \, do  = \int_\Gamma |\grad f|^2 \, do \neq 0.
\end{align*}
This is a contradiction and it finally follows that
$$
	\| \psi \|_{H^1}^2 \leq C(\Gamma, f) \left( 2 \int_\Gamma |\grad f|^2 \psi_\nu^2 \, do + \int_\Gamma |\grad \psi_\nu |^2 + |\grad \psi_\tau |^2 \, do \right)
	= C(\Gamma, f) b(\psi, \psi).
$$
\end{proof}
\begin{lem}
\label{lemma_bartels}
Let $\theta_d(h) := h^{1 - d/2} \log h^{-1}$ if $d \geq 2$ and $\theta_d(h) := 1$ if $d = 1$.
If $h < 1$, then the following estimate holds
$$
	\| f_h^l \|_{L^\infty} \leq C \theta_d(h) \| f_h^l \|_{H^1},
$$
for all $f_h^l \in V_h^l$ on the $d$-dimensional hypersurface $\Gamma \subset \mathbb{R}^{d+1}$.
\end{lem}
\begin{proof}
For $d \geq 2$ the claim follows from Lemma 1.4.9 in \cite{Ba08}, the equivalence of norms in Proposition \ref{prop_equivalence_of_norms} and the fact that we have
assumed that the mesh $\Gamma_h$ is quasi-uniform. For $d = 1$ the embedding of $H^1(\Gamma)$ into $L^\infty(\Gamma)$ gives the result. 
\end{proof}
It is easy to see that for $d \leq 3$ we have
\begin{equation}
\label{limit_h_theta}
	\lim_{h \searrow 0} h \ \theta_d(h) = 0.
\end{equation}
We define the following finite element spaces
$$
	X_h := \left\{ \psi_h \in (V_h)^2 : \int_\Gamma (\psi_h^l)_\tau \, do = 0 \right\} \quad \textnormal{and} \quad X_h^l := \left\{ \psi_h^l : \psi_h \in  X_h \right\}
	\subset X.
$$
\begin{prop}
\label{prop_conditions_inverse_function_theorem}
Let $h_0 > 0$ be sufficiently small and $h \leq h_0$. Furthermore, let $f: \Gamma \rightarrow \mathbb{S}^1 \subset \mathbb{R}^2$ be a smooth harmonic map
on the closed hypersurface $\Gamma \subset \mathbb{R}^{d+1}$, $d \leq 3$. Then the following estimates hold
\begin{align}
& |E_h'(I_h f_l)(\psi_h)| \leq C(\Gamma,f,h_0) h \| \psi_h \|_{H^1} \quad \textnormal{for all $\psi_h \in (V_h)^2$,}
\label{first_estimate}
\\
& E''_h(I_h f_l)(\psi_h, \psi_h) \geq \tfrac{\lambda^\ast}{2} \| \psi_h \|_{H^1}^2 
	\quad \textnormal{for all $\psi_h \in X_h$,} 
\label{second_estimate}	
\end{align}
and 	for $\eta_h \in (V_h)^2$ with $\| \eta_h \|_{H^1} \leq \tfrac{\min\{ \omega_0 , \lambda^\ast/4 \}}{C(\Gamma, f, h_0) \theta_d(h)}$ we have
\begin{align}
& |(E''_h(I_h f_l) - E''_h(I_h f_l + \eta_h))(\psi_h, \psi_h)| \leq \tfrac{\lambda^\ast}{4} \| \psi_h \|_{H^1}^2 
	\quad \textnormal{for all $\psi_h \in (V_h)^2$.}
\label{third_estimate}	
\end{align}
Here, $\lambda^\ast := \tfrac{\lambda}{C(\Gamma)}$ with the coercivity constant $\lambda$ from Lemma \ref{lemma_coercivity} 
and $\omega_0$ is the constant from Lemma \ref{lemma_continuity_second_variation}. 
\end{prop}
\begin{proof}
The first inequality directly follows from Lemma \ref{lemma_smallness_of_first_variation}. For the second estimate we use Lemmas \ref{lemma_coercivity} and
\ref{lemma_consistency_second_variation} and obtain
\begin{align*}
E_h''(I_h f_l)(\psi_h, \psi_h) &\geq E''(f)(\psi_h^l, \psi_h^l) - |E''(f)(\psi_h^l, \psi_h^l) - E_h''(I_h f_l)(\psi_h, \psi_h)|
 \\
 &\geq b(\psi_h^l, \psi_h^l) - C(\Gamma, f, h_0) h \| \psi_h \|_{H^1}^2
 \\
 &\geq \lambda \| \psi_h^l \|_{H^1}^2 - C(\Gamma, f, h_0) h \| \psi_h \|_{H^1}^2
 \\
 &\geq \left( \tfrac{\lambda}{C(\Gamma)} - C(\Gamma, f, h_0) h \right) \| \psi_h \|_{H^1}^2.
\end{align*} 
If we choose $h \leq h_0$ for $h_0$ sufficiently small the claim follows. For the last inequality we observe that
$$
	\| \eta_h \|_{L^\infty} \leq C(\Gamma) \| \eta_h^l \|_{L^\infty} \leq C(\Gamma) \theta_d(h) \| \eta_h^l \|_{H^1} \leq \omega_0.
$$
We can therefore apply Lemma \ref{lemma_continuity_second_variation}. This yields
\begin{align*}
	|(E''_h(I_h f_l) - E''_h(I_h f_l + \eta_h))(\psi_h, \psi_h)| &\leq C(f, h_0) \theta_d(h) \| \eta_h \|_{H^1} \| \psi_h \|^2_{H^1} \sum_{k=0}^4 \theta_d(h)^k \| \eta_h \|_{H^1}^k
	\\
	&\leq C(\Gamma,f,h_0) \theta_d(h) \| \eta_h \|_{H^1} \| \psi_h \|^2_{H^1}
	\\
	&\leq \tfrac{\lambda^\ast}{4} \| \psi_h \|^2_{H^1}.
\end{align*}
\end{proof}

\begin{thm}
Let $d \leq 3$. Furthermore, let $h_0 > 0$ be sufficiently small and $h \leq h_0$. 
Then for every smooth harmonic map $f: \Gamma \rightarrow \mathbb{S}^1 \subset \mathbb{R}^2$ on the closed hypersurface $\Gamma \subset \mathbb{R}^{d+1}$ there exists $f_h \in (V_h)^2$
with $\int_\Gamma (f_h^l)_\tau \, do = \int_\Gamma (I_h^l f)_\tau \, do$ such that
\begin{equation}
	\int_{\Gamma_h} \gradh f_h : \gradh \psi_h \left( \tfrac{1}{2} + \tfrac{1}{2|f_h|^4} \right) \, do
		= \int_{\Gamma_h} f_h \cdot \psi_h \tfrac{|\gradh f_h|^2}{|f_h|^6} \, do
	\label{discrete_stationary_equation}	
\end{equation}
for all $\psi_h \in (V_h)^2$, and
$$
	\| f - f_h^l \|_{H^1} \leq C(\Gamma,f,h_0) h.
$$
Moreover, $f_h \in (V_h)^2$ is the unique stationary point for $E_h(f_h) = \tfrac{1}{2} \int_{\Gamma_h} | \gradh f_h|^2 \big(\tfrac{1}{2} + \tfrac{1}{2 |f_h|^4} \big) do$ which satisfies the two conditions 
$\int_\Gamma (f_h^l)_\tau \, do = \int_\Gamma (I_h^l f)_\tau \, do$
and 
$$
	\| f - f_h^l \|_{H^1} \leq \widetilde{C}(\Gamma,f,h_0) \frac{1}{\theta_d(h)}.
$$
\end{thm}
\begin{proof}
We set $\mathcal{X} := \{ \psi_h \in (V_h)^2 : \int_\Gamma (\psi_h^l)_\tau \, do = \int_\Gamma (I_h^l f)_\tau \, do \}$ 
and define $Y$ to be the dual space of $X_h$, that is $Y := X_h^\ast$. The tangent space $\hat{X}$ of $\mathcal{X}$ is given by $\hat{X} = X_h$. 
On $\mathcal{X}$ and $\hat{X}$ we use the norm $\| \cdot \|_{X_h} := \| \cdot \|_{H^1}$ and on $Y$ 
the corresponding dual norm $\| \cdot \|_Y := \| \cdot \|_{\ast} := \| \cdot \|_{L(X_h, \mathbb{R})}$. 
Since $\mathcal{X}$, $\hat{X}$ and $Y$ are finite-dimensional, they are also (affine) Banach spaces.
The map $\mathcal{F}: \mathcal{X} \rightarrow Y$ is defined by $\mathcal{F}(x) := E'_h(x)$. 
We now choose $x_0 := I_h f_l$, $\delta := C(\Gamma, f, h_0) h$, $\alpha = \tfrac{\lambda^\ast}{2}$ as well as $\beta := \tfrac{\lambda^\ast}{4}$
with $\lambda^\ast$ from the above proposition.
The first condition in Lemma \ref{lemma_inverse_function_theorem}, that is $\| \mathcal{F}(x_0) \|_Y < \delta$ or respectively,
$\| E'_h (I_f f_l)\|_\ast \leq C(\Gamma, f, h_0) h$, is then satisfied because of (\ref{first_estimate}).
We next observe that
$$
	\| \mathcal{F}'(x_0)^{-1} \|_{L(Y,\hat{X})} \leq \left( \inf_{x \in \hat{X}} \frac{\| \mathcal{F}'(x_0)(x) \|_Y }{\| x \|_{\hat{X}}}\right)^{-1} 
	= \left( \inf_{\psi_h \in X_h} \frac{\| \mathcal{F}'(I_h f_l)(\psi_h) \|_\ast }{\| \psi_h \|_{H^1}}\right)^{-1}.
$$
Note that $\mathcal{F}'(x_0) \in L(\hat{X},Y) = L(\hat{X}, L(\hat{X}, \mathbb{R}))$ and $E''_h(x_0) \in L(\hat{X} \times \hat{X}, \mathbb{R})$.
With the identification $\mathcal{F}'(x_0)(\psi_h)(\phi_h) = E''_h(x_0) (\psi_h, \phi_h)$, we obtain
$$
	\frac{\| \mathcal{F}'(I_h f_l) (\psi_h) \|_\ast}{\| \psi_h \|_{H^1}} \geq \frac{|\mathcal{F}' (I_h f_l)(\psi_h)(\psi_h)|}{\| \psi \|_{H^1}^2}
	= \frac{|E''_h(I_h f_l)(\psi_h, \psi_h)|}{\| \psi \|_{H^1}^2} \geq \frac{\lambda^\ast}{2}, 
$$
where we used (\ref{second_estimate}) in the last step. Hence, $\| \mathcal{F}'(x_0)^{-1} \|_{L(Y,\hat{X})} \leq ( \tfrac{\lambda^\ast}{2} )^{-1} = \alpha^{-1}$.
Furthermore, the above identification gives
\begin{align*}
	&\| \mathcal{F}'(x_0 + \eta_h) - \mathcal{F}'(x_0) \|_{L(\hat{X},Y)} 
	= \sup_{\psi_h \in X_h} \sup_{\phi_h \in X_h} \frac{|(E''_h(I_h f_l + \eta_h) - E''_h(I_h f_l))(\psi_h, \phi_h)|}{\| \psi_h \|_{H^1} \| \phi_h \|_{H^1}}
	\\
	&= \sup_{\psi_h \in X_h} \sup_{\phi_h \in X_h} \tfrac{1}{4} \left| (E''_h(I_h f_l + \eta_h) - E''_h(I_h f_l))
		(\tfrac{\psi_h}{\| \psi_h \|_{H^1}} + \tfrac{\phi_h}{\| \phi_h \|_{H^1}}, \tfrac{\psi_h}{\| \psi_h \|_{H^1}} + \tfrac{\phi_h}{\| \phi_h \|_{H^1}})  
		\right.
		\\ 
		& \qquad
		\left.
			- (E''_h(I_h f_l + \eta_h) - E''_h(I_h f_l))
		(\tfrac{\psi_h}{\| \psi_h \|_{H^1}} - \tfrac{\phi_h}{\| \phi_h \|_{H^1}}, \tfrac{\psi_h}{\| \psi_h \|_{H^1}} - \tfrac{\phi_h}{\| \phi_h \|_{H^1}}) 
		\right|,
\end{align*}
where we have used polarization in the second step. Applying (\ref{third_estimate}) leads to
\begin{align*}
	&\| \mathcal{F}'(x_0 + \eta_h) - \mathcal{F}'(x_0) \|_{L(\hat{X},Y)} 
	\\	
	&\leq \sup_{\psi_h \in X_h} \sup_{\phi_h \in X_h} \tfrac{1}{4} \left(
		\tfrac{\lambda^\ast}{4} \| \tfrac{\psi_h}{\| \psi_h \|_{H^1}} + \tfrac{\phi_h}{\| \phi_h \|_{H^1}} \|_{H^1}^2
		+ \tfrac{\lambda^\ast}{4} \| \tfrac{\psi_h}{\| \psi_h \|_{H^1}} - \tfrac{\phi_h}{\| \phi_h \|_{H^1}} \|_{H^1}^2
	\right)
	= \tfrac{\lambda^\ast}{4} = \beta,
\end{align*}
if $\eta_h \in X_h$ satisfies $\| \eta_h \|_{H^1} \leq \tfrac{\min\{ \omega_0 , \lambda^\ast/4 \}}{C(\Gamma, f, h_0) \theta_d(h)}$.
Obviously, we have $\beta < \alpha$ and the condition $\delta < (\alpha - \beta) \varepsilon$ in Lemma \ref{lemma_inverse_function_theorem}
leads to the condition $C(\Gamma, f, h_0) h \leq \tfrac{\lambda^\ast}{4} \varepsilon$. In order to apply the Inverse Function Theorem in Lemma \ref{lemma_inverse_function_theorem},
we therefore have to choose $\varepsilon > 0$ such that
\begin{equation*}
	\varepsilon \leq \tfrac{\min\{ \omega_0 , \lambda^\ast/4 \}}{C(\Gamma, f, h_0) \theta_d(h)} \quad \textnormal{and} \quad C(\Gamma, f, h_0) h \leq \tfrac{\lambda^\ast}{4} \varepsilon. 
\end{equation*}
We set $\varepsilon := \tfrac{4 C(\Gamma, f, h_0)}{\lambda^\ast} h$, which obviously satisfies the second condition. The first condition follows from
(\ref{limit_h_theta}). We conclude that there exists a $f_h \in \mathcal{X}$ with $\| I_h f_l - f_h \|_{H^1} \leq \varepsilon = C(\Gamma, f, h_0) h$ 
such that (\ref{discrete_stationary_equation}) holds for all $\psi_h \in \hat{X} = X_h$. Using the interpolation estimates (\ref{interpolation_estimate_1}) and (\ref{interpolation_estimate_2})
as well as the equivalence of norms in Proposition \ref{prop_equivalence_of_norms}, we obtain $\| f - f_h^l \|_{H^1} \leq C(\Gamma, f, h_0) h$.

By choosing $\varepsilon = \tfrac{\min\{ \omega_0 , \lambda^\ast/4 \}}{C(\Gamma, f, h_0) \theta_d(h)}$, which satisfies both conditions because of (\ref{limit_h_theta}),
we see that the above solution $f_h$ is the unique point in $\mathcal{X}$ satisfying (\ref{discrete_stationary_equation}) and 
$$
	\| I_h f - f_h \|_{H^1} \leq \widetilde{C}(\Gamma, f, h_0) \frac{1}{\theta_d(h)},
	\quad \textnormal{or respectively,} \quad \| f - f^l_h \|_{H^1} \leq \widetilde{C}(\Gamma, f, h_0) \frac{1}{\theta_d(h)}.
$$
It remains to show that (\ref{discrete_stationary_equation}) holds for all $\psi_h \in (V_h)^2$ if it holds for all $\psi_h \in X_h$.
Let $\tilde{\varphi}_h \in (V_h)^2$ be such that $\int_\Gamma (\tilde{\varphi}_h^l)_\tau \, do \neq 0$, then for dimensional reasons it holds that 
$$
	(V_h)^2 = X_h \oplus span\{ \tilde{\varphi}_h \}.
$$
For $f_h^\perp \in (V_h)^2$ with $(f_{h1}^\perp, f_{h2}^\perp) = ( - f_{h2}, f_{h1})$ we have 
\begin{align*}
	|\int_\Gamma (f_h^{\perp l})_\tau \, do| &\geq \int_\Gamma |f^\perp|^2 \, do - |\int_\Gamma (f_h^{\perp l} - f^\perp) \cdot f^\perp \, do|
	\\
	&\geq |\Gamma| - |\Gamma|^{1/2} \| f - f_h^l \|_{L^2}
	\\
	&\geq |\Gamma| - |\Gamma|^{1/2} C(\Gamma,f , h_0) h > 0
\end{align*}
for $h_0$ sufficiently small. We can therefore choose $\tilde{\varphi}_h = f_h^\perp$, that is $(V_h)^2 = X_h \oplus span\{ f_h^\perp \}$. 
Since $f_h \cdot f_h^\perp = 0$ and $\gradh f_h : \gradh f_h^\perp = 0$, equation (\ref{discrete_stationary_equation}) also holds
for $\psi_h = f_h^\perp$. This implies that (\ref{discrete_stationary_equation}) is true for all $\psi_h \in (V_h)^2$.
\end{proof}

\section{Numerical results}
\subsection{Implementation}
\label{section_implementation}
In order to apply Algorithm \ref{algorithm_sphere}, we have to solve a linear system of equations. Let $\phi_i$ for $i = 1, \ldots, N_V$
denote the Lagrange basis functions corresponding to the $i$-th vertex of the polyhedral hypersurface $\Gamma_h$. Furthermore, let
$\{ e_1, \ldots, e_{n+1} \}$ be the canonical basis of $\mathbb{R}^{n+1}$. The components of $f_h^{m+1}$ with respect to $\{ \phi_i e_\alpha \}$ are
denoted by $\mathbf{f}^{i \alpha}$, that is $f_h^{m+1} = \sum_{i=1}^{N_V} \sum_{\alpha = 1}^{n+1} \mathbf{f}^{i \alpha} \phi_i e_\alpha$.
In the $m$-th time step one then has to solve the system
\begin{align}
\label{linear_system}
 \sum_{j=1}^{N_V} \sum_{\beta=1}^{n+1} \left( \tfrac{1}{\tau} \mathbf{M}_{ij \alpha \beta} + \mathbf{S}_{ij \alpha \beta} \right) \mathbf{f}^{j\beta} = \mathbf{b}_{i\alpha},
\end{align}
for all $i = 1, \ldots, N_V$ and $\alpha = 1, \ldots, n+1$, where
\begin{align*}
 &	\mathbf{M}_{ij \alpha \beta} := \delta_{\alpha \beta} \int_{\Gamma_h} \phi_i \phi_j \left( \tfrac{1}{2} + \tfrac{1}{2|f_h^m|^4} \right)  \, do,
 \\
 &  \mathbf{S}_{ij \alpha \beta} := \delta_{\alpha \beta} \int_{\Gamma_h} \gradh \phi_i \cdot \gradh \phi_j \left( \tfrac{1}{2} + \tfrac{1}{2|f_h^m|^4} \right)  \, do,
 \\
 &  \mathbf{b}_{i \alpha} := \sum_{j=1}^{N_V} \sum_{\beta=1}^{n+1} \left( \tfrac{1}{\tau} \mathbf{M}_{ij \alpha \beta} 
 						+ \delta_{\alpha \beta} \int_{\Gamma_h} \phi_i \phi_j \tfrac{| \gradh f_h^m |^2}{|f_h^m|^6} \, do \right) \mathbf{f}^{j\beta}_{old},
\end{align*}
and $f_h^{m} = \sum_{i=1}^{N_V} \sum_{\alpha = 1}^{n+1} \mathbf{f}^{i \alpha}_{old} \phi_i e_\alpha$.
Note that the integrands in the above integrals are not constant on each simplex. 
For this reason a quadrature rule which is exact for a sufficiently high polynomial degree, see \cite{SS05} Section 1.4.6, should be employed.
In our tests, a quadrature rule exact for polynomials of degree five or less seemed to be sufficient. 
The implementation of Algorithm \ref{algorithm_sphere} is straightforward, since only the two weights $\left( \tfrac{1}{2} + \tfrac{1}{2 |f_h^m|^4} \right)$
and $\tfrac{|\gradh f_h^m|^2}{|f_h^m|^6}$ have to be included in the assemblage of the usual mass and stiff matrices.  
The linear system (\ref{linear_system}) can be solved by a conjugate gradient method.
The following numerical experiments were performed within the Finite Element Toolbox ALBERTA, see \cite{SS05}.

\subsection{Numerical examples}
\label{section_numerical_examples}

\subsubsection*{Experiment 1}
We compute a triangulation of the unit sphere by $n=4,5,6$ global refinements of an octahedron. In each refinement step all simplices are bisected twice. 
$f_h^0$ is then initialised as a linear interpolation of the identity map on the unit sphere. After that, the unit sphere is deformed by applying the map 
$\mathbb{S}^2 \ni (x_1,x_2,x_3) \mapsto (x_1, \alpha(x_1) x_2, \alpha(x_1) x_2)$ with $\alpha(x_1) = 0.6 \, x_1^2 + 0.4$. This leads to the polyhedral surface $\Gamma_h$,
see left picture in Figure \ref{Fig_Ex_1}. 
Since the values of $f_h^0$ have not been changed, $f_h^0$ is now a map from $\Gamma_h$ into the unit sphere. We choose the time step size $\tau = 0.001$
and compute the harmonic map heat flow by Algorithm \ref{algorithm_sphere}. 
As a stopping criterion we use 
$$
	\textnormal{Stop if:} \quad  \max_{p_j \in V} \tfrac{1}{\tau} |f_h^{m+1}(p_j) - f_h^m(p_j)| \leq\ 10^{-5},
$$
where $V$ is the set of the mesh vertices of $\Gamma_h$.
The experimental results are presented in Figures \ref{Fig_Ex_1} and \ref{Distance_Fig_Ex_1}.
We have studied the behaviour of the maximal distance
$$
	\max_{p_j \in V} ||f_h^m(p_j)| - 1|.
$$
We observe that this quantity remains bounded in time, although it is not zero for $t > 0$. 
We emphasize that one cannot expect the maximal distance to be zero for two reasons.
First, in our approach the mesh vertices are not distinguished from any other points on the surface and second, a piecewise linear surface cannot lie entirely in the unit sphere.
The important outcome of this experiment is that the maximal distance to the unit sphere remains bounded in time and that it decreases when the mesh size $h$ is reduced by global refinements, 
see \mbox{Figure \ref{Distance_Fig_Ex_1}}.
\begin{figure}
\begin{center}
\raisebox{-0.5\height}{\includegraphics[scale= 0.19]{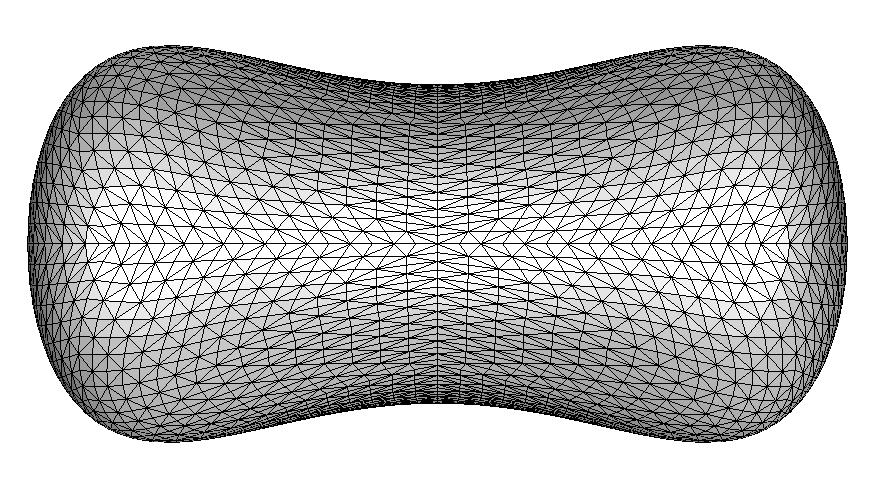}}
\raisebox{-0.5\height}{\includegraphics[scale= 0.2]{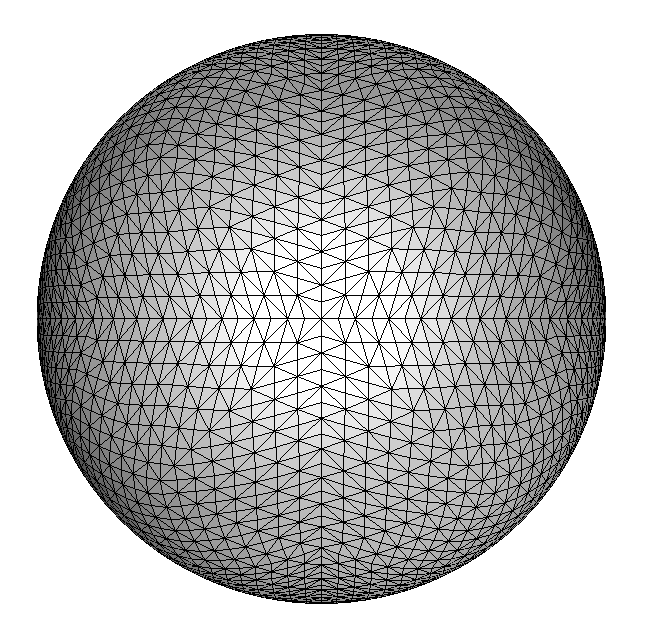}}
\raisebox{-0.5\height}{\includegraphics[scale= 0.2]{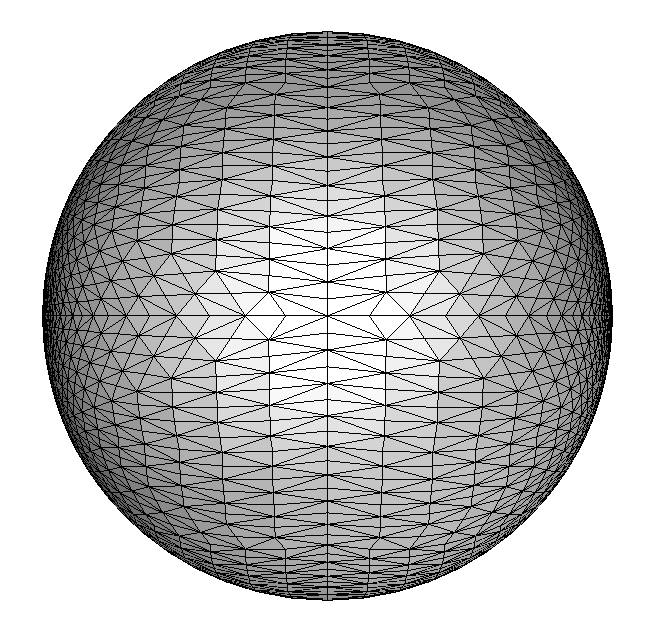}}
\end{center}
\caption{Computation of the harmonic map heat flow $f_h^m: \Gamma_h \rightarrow \mathbb{R}^3$ into the two-dimensional sphere. The left picture shows the polyhedral surface $\Gamma_h$
for $n=5$ global refinements. The pictures in the middle and on the right show the image $f_h^m(\Gamma_h)$ for time $t = 0$ and $t \approx 1.9$, respectively. See Experiment $1$ for further details.}
\label{Fig_Ex_1}
\end{figure}
\gdef\gplbacktext{}%
\gdef\gplfronttext{}%
\begin{figure}
\begin{center}
  \begin{picture}(3968.00,2834.00)%
    \gplgaddtomacro\gplbacktext{%
      \csname LTb\endcsname%
      \put(396,110){\makebox(0,0)[r]{\strut{} 0}}%
      \csname LTb\endcsname%
      \put(396,496){\makebox(0,0)[r]{\strut{} 0.002}}%
      \csname LTb\endcsname%
      \put(396,882){\makebox(0,0)[r]{\strut{} 0.004}}%
      \csname LTb\endcsname%
      \put(396,1268){\makebox(0,0)[r]{\strut{} 0.006}}%
      \csname LTb\endcsname%
      \put(396,1653){\makebox(0,0)[r]{\strut{} 0.008}}%
      \csname LTb\endcsname%
      \put(396,2039){\makebox(0,0)[r]{\strut{} 0.01}}%
      \csname LTb\endcsname%
      \put(396,2425){\makebox(0,0)[r]{\strut{} 0.012}}%
      \csname LTb\endcsname%
      \put(396,2811){\makebox(0,0)[r]{\strut{} 0.014}}%
      \csname LTb\endcsname%
      \put(528,-110){\makebox(0,0){\strut{}0.0}}%
      \csname LTb\endcsname%
      \put(1670,-110){\makebox(0,0){\strut{}0.5}}%
      \csname LTb\endcsname%
      \put(2811,-110){\makebox(0,0){\strut{}1.0}}%
      \csname LTb\endcsname%
      \put(3953,-110){\makebox(0,0){\strut{}1.5}}%
      \put(-638,1460){\rotatebox{-270}{\makebox(0,0){\strut{}Maximal distance to unit sphere}}}%
      \put(2240,-440){\makebox(0,0){\strut{}Time}}%
    }%
    \gplgaddtomacro\gplfronttext{%
      \csname LTb\endcsname%
      \put(2966,1680){\makebox(0,0)[r]{\strut{}$n=4$}}%
      \csname LTb\endcsname%
      \put(2966,1460){\makebox(0,0)[r]{\strut{}$n=5$}}%
      \csname LTb\endcsname%
      \put(2966,1240){\makebox(0,0)[r]{\strut{}$n=6$}}%
    }%
    \gplbacktext
    \put(0,0){\includegraphics{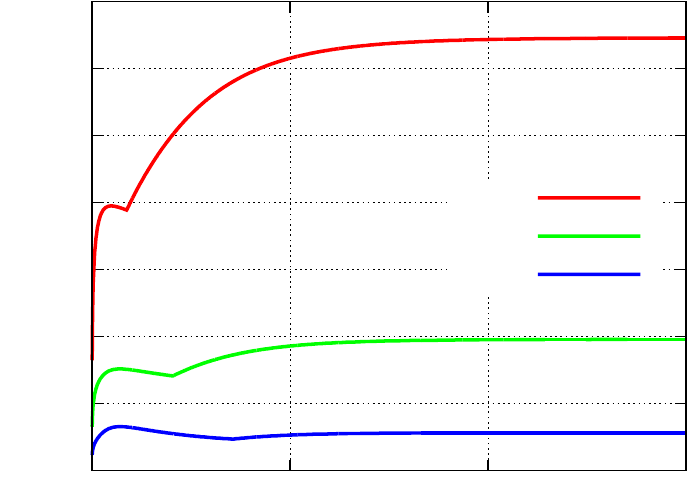}}%
    \gplfronttext
  \end{picture}%
\end{center}  
\vspace*{8mm}
\caption{The picture shows the maximal distance of the image $f_h^m(p_j)$ to the unit sphere, that is $\max_{p_j \in V} ||f_h^m(p_j)| - 1|$,
for different global mesh refinements $n$. See Experiment $1$ for further details.}
\label{Distance_Fig_Ex_1}
\end{figure}
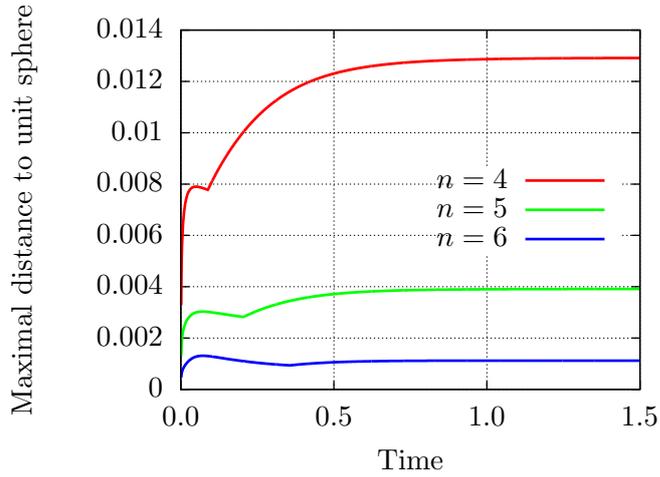

\subsubsection*{Experiment 2}
In this experiment we study the behaviour of Algorithm \ref{algorithm_sphere} in the case that the initial map $f_h^0$ does not map into the unit sphere.
This is of course against the way how one should normally use our numerical scheme. In a standard application of the scheme, one would try to find an initial map $f_h^0$ 
which approximates the target manifold in the best possible way. Here, we do the opposite to demonstrate the performance of our algorithm. 
We therefore compose our original map $f_h^0$ from the Experiment $1$ with the map $(y_1,y_2,y_3) \mapsto \beta(y_1,y_3)(y_1,y_2,y_3)$, where $\beta(y_1,y_3) =  0.5 + y_1^2 y_3^2$.
The image $f_h^0(\Gamma_h)$ of our new initial map $f_h^0$ is shown in Figure \ref{Fig_Ex_2}. The maximal distance of the map $f_h^0$ to the target manifold is $0.5$.
In Figure \ref{Distance_Fig_Ex_2}, this maximal distance decays rapidly (exponentially) in time. This behaviour can be understood from equation (\ref{heat_equation_distance_function}), which says
that the distance to the target manifold decreases monotonically (in the continuous case). In particular, it strictly decreases as long as the gradient of $f$ is not zero.
At time $t \approx 2.8$, $f_h^m$ therefore maps approximately into the unit sphere, see right picture in Figure \ref{Fig_Ex_2}.  
\begin{figure}
\begin{center}
\raisebox{-0.5\height}{\includegraphics[scale= 0.2]{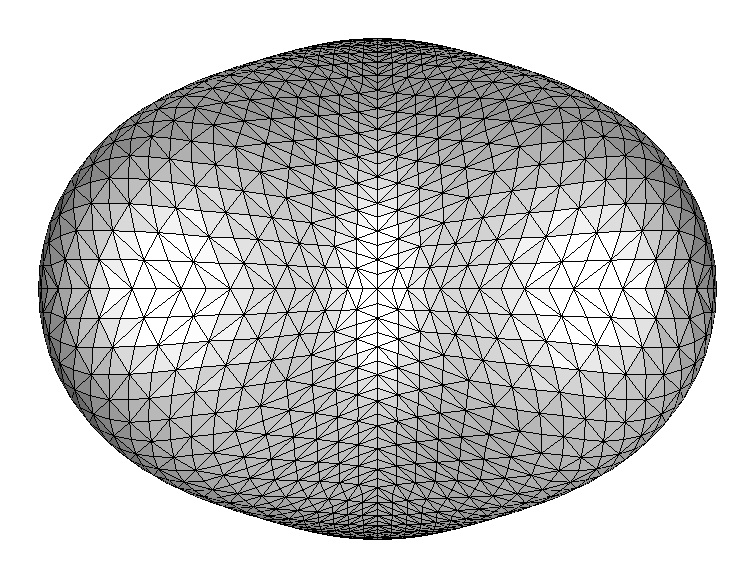}}
\raisebox{-0.5\height}{\includegraphics[scale= 0.2]{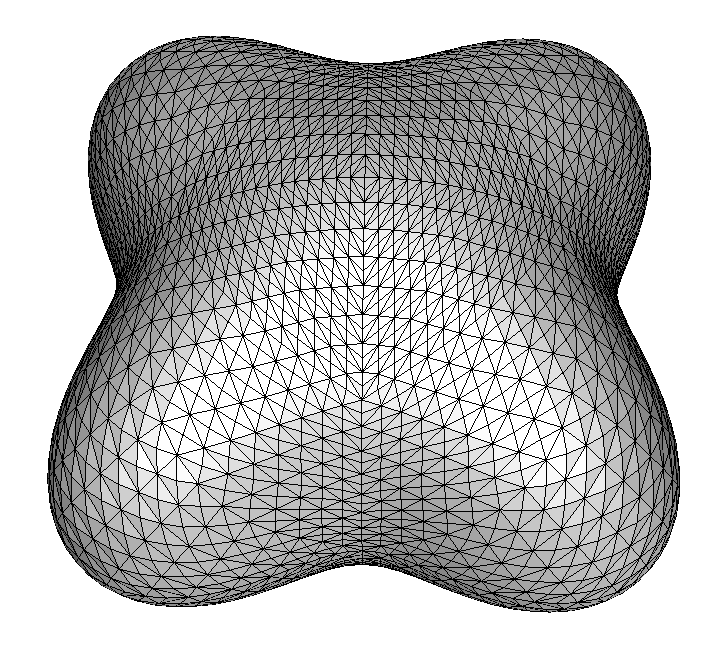}}
\raisebox{-0.5\height}{\includegraphics[scale= 0.2]{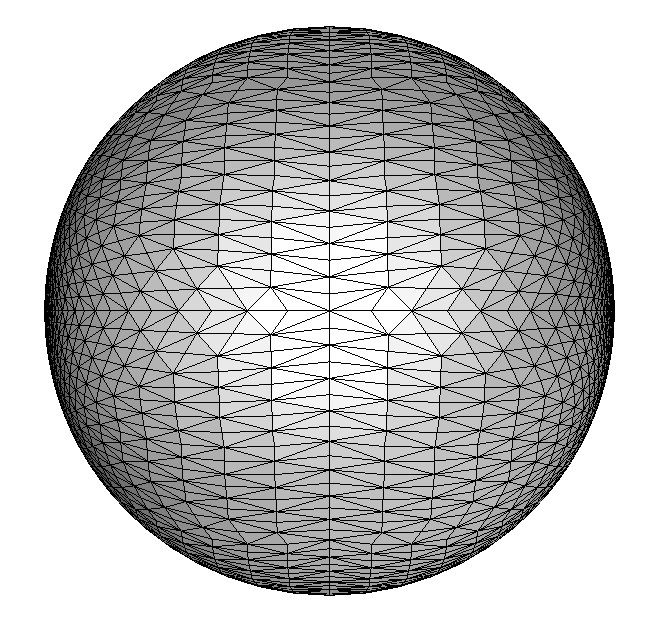}}
\end{center}
\caption{The pictures on the left and in the middle show the image $f_h^0(\Gamma_h)$ at time $t=0$. $\Gamma_h$ is shown in Fig. \ref{Fig_Ex_1}.
The initial map $f_h^0$ does not map into the unit sphere. The maximal distance of $f_h^0(p_j)$ to the unit sphere at time $t=0$ is $0.5$, see Fig. \ref{Distance_Fig_Ex_2}. 
The right pictures shows the image $f_h^m(\Gamma_h)$ at time $t \approx 2.8$. At that time, $f_h^m$ maps approximately into the unit sphere. See Experiment $2$ for further details.
}
\label{Fig_Ex_2}
\end{figure}
\gdef\gplbacktext{}%
\gdef\gplfronttext{}%
\begin{figure}
\begin{center}
  \begin{picture}(3968.00,2834.00)%
    \gplgaddtomacro\gplbacktext{%
      \csname LTb\endcsname%
      \put(396,110){\makebox(0,0)[r]{\strut{} 0}}%
      \csname LTb\endcsname%
      \put(396,650){\makebox(0,0)[r]{\strut{} 0.1}}%
      \csname LTb\endcsname%
      \put(396,1190){\makebox(0,0)[r]{\strut{} 0.2}}%
      \csname LTb\endcsname%
      \put(396,1731){\makebox(0,0)[r]{\strut{} 0.3}}%
      \csname LTb\endcsname%
      \put(396,2271){\makebox(0,0)[r]{\strut{} 0.4}}%
      \csname LTb\endcsname%
      \put(396,2811){\makebox(0,0)[r]{\strut{} 0.5}}%
      \csname LTb\endcsname%
      \put(528,-110){\makebox(0,0){\strut{}0.0}}%
      \csname LTb\endcsname%
      \put(1213,-110){\makebox(0,0){\strut{}0.5}}%
      \csname LTb\endcsname%
      \put(1898,-110){\makebox(0,0){\strut{}1.0}}%
      \csname LTb\endcsname%
      \put(2583,-110){\makebox(0,0){\strut{}1.5}}%
      \csname LTb\endcsname%
      \put(3268,-110){\makebox(0,0){\strut{}2.0}}%
      \csname LTb\endcsname%
      \put(3953,-110){\makebox(0,0){\strut{}2.5}}%
      \put(-374,1460){\rotatebox{-270}{\makebox(0,0){\strut{}Maximal distance to unit sphere}}}%
      \put(2240,-440){\makebox(0,0){\strut{}Time}}%
    }%
    \gplgaddtomacro\gplfronttext{%
      \csname LTb\endcsname%
      \put(2966,2638){\makebox(0,0)[r]{\strut{}$n=5$}}%
    }%
    \gplbacktext
    \put(0,0){\includegraphics{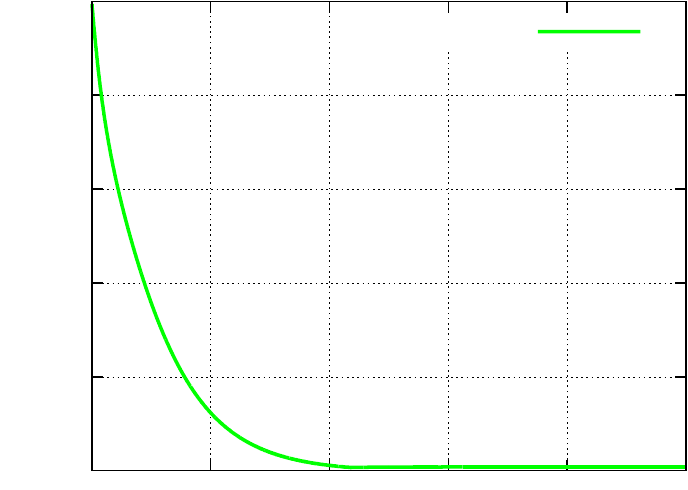}}%
    \gplfronttext
  \end{picture}%
\end{center}
\vspace*{8mm}
\caption{Time development of the maximal distance $\max_{p_j \in V} ||f_h^m(p_j)| - 1|$ under Algorithm \ref{algorithm_sphere}.
The initial map $f_h^0$ is visualized in Fig. \ref{Fig_Ex_2}. The distance to the unit sphere tends to zero rapidly and $f_h^m$ approximately maps into the unit sphere after some time. 
See Experiment $2$ for further details.}
\label{Distance_Fig_Ex_2}
\end{figure}
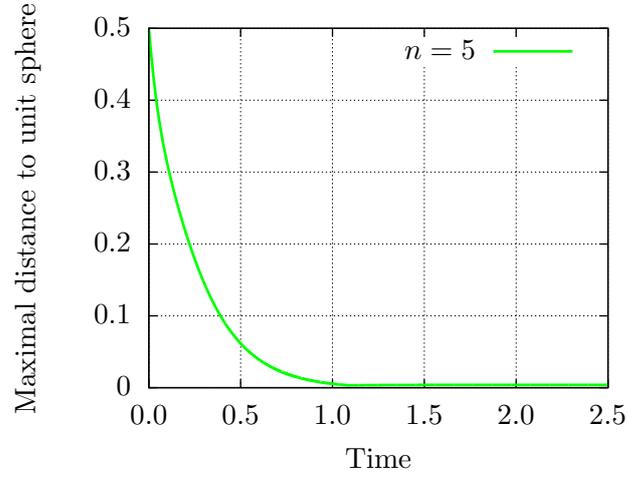

\subsubsection*{Experiment 3}
In the last experiment, we change the surface $\Gamma_h$ by deforming the unit sphere according to  
$\mathbb{S}^2 \ni (x_1,x_2,x_3) \mapsto (x_1, \alpha(x_1) x_2, \alpha(x_1) x_2)$ with $\alpha(x_1) = 0.75 \, x_1^2 + 0.25$, see Experiment 1 for details.
The initial map $f_h^0$ maps the mesh vertices $p_j \in \Gamma_h$ into the unit sphere like in Experiment $1$.
The result is presented in Figure \ref{Fig_Ex_3}.
\begin{figure}
\begin{center}
\raisebox{-0.5\height}{\includegraphics[scale= 0.2]{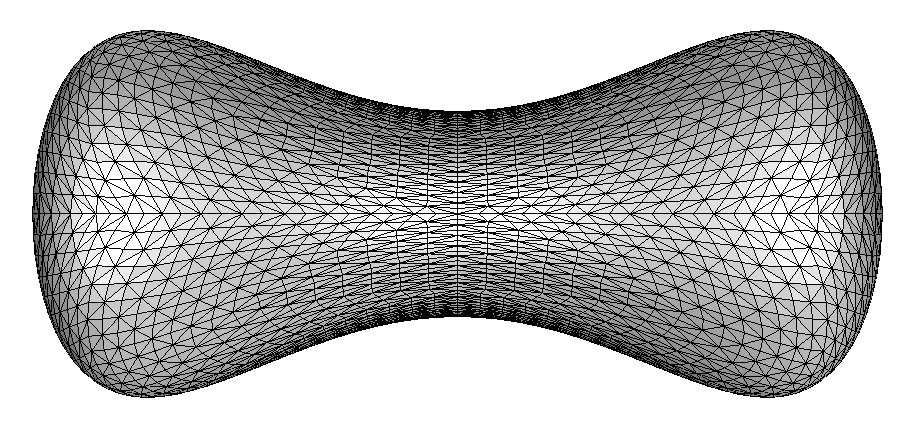}}
\hspace*{8mm}
\raisebox{-0.5\height}{\includegraphics[scale= 0.2]{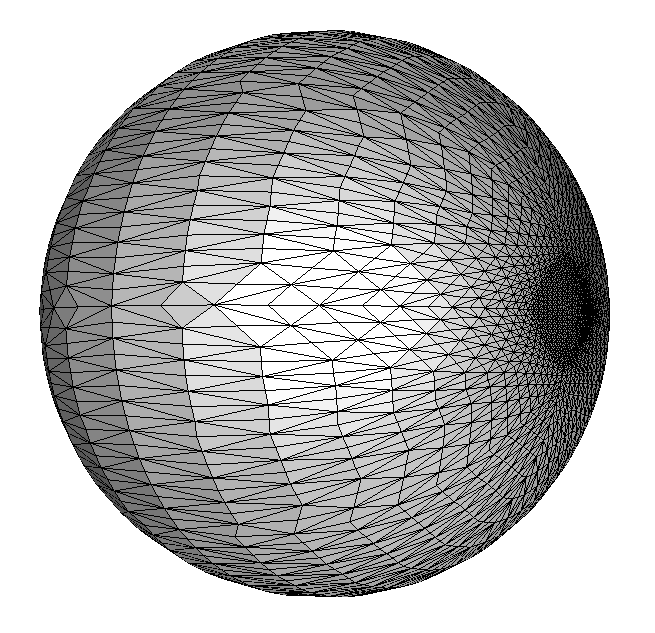}}
\end{center}
\caption{Computation of the harmonic map heat flow $f_h^m: \Gamma_h \rightarrow \mathbb{R}^3$ into the two-dimensional sphere. The left picture shows the polyhedral hypersurface $\Gamma_h$
for $n= 5$ global refinements. The initial map $f_h^0$ was chosen like in Experiment 1.
The picture on the right shows the image $f_h^m(\Gamma_h)$ at time $t \approx 2.6$. See Experiment $3$ for further details.}
\label{Fig_Ex_3}
\end{figure}

\section*{Appendix}
\begin{lem}
\label{lemma_distance_function}
Let $(\Gamma,m)$ be a smooth closed Riemannian manifold and $\M \subset \mathbb{R}^{n+1}$ an orientable smooth closed hypersurface in $\mathbb{R}^{n+1}$.
Let $\N_T$ be a tubular neighbourhood of $\M$ of fixed width such that the decomposition $x = a(x) + d(x) Dd(x)$ with $a(x) \in \M$ is unique on $\N_T$.
Furthermore, let $G$ be the extended metric defined in (\ref{extended_metric_hypersurfaces}). Finally, let $f: \Gamma \rightarrow \N_T$ be a differentiable map. Then the following identity holds 
$$
	m(grad_m f^\alpha, grad_m f^\beta) (\nabla^G_\alpha \nabla^G_\beta d) \circ f
	= m(grad_m f^\alpha, grad_m f^\beta) \, d \circ f \left( D_\alpha D_\rho d \, D_\beta D_\rho d \right) \circ f	.	
$$
\end{lem}
\begin{proof}
From (\ref{properties_signed_distance_function}) and (\ref{extended_metric_hypersurfaces}) it follows that
$$
	G(x) Dd(x) = Dd(x), \quad \textnormal{and hence,} \quad Dd(x) = G^{-1}(x) Dd(x) 
$$
We then compute by using (\ref{properties_signed_distance_function}) and (\ref{extended_metric_hypersurfaces}) again that
\begin{align*}
	&m^{ij} \partial_i f^\alpha \partial_j f^\beta (\nabla^G_\alpha \nabla^G_\beta d) \circ f
	= m^{ij} \partial_i f^\alpha \partial_j f^\beta \left[ D_\alpha D_\beta d - \Gamma(G)^\gamma_{\alpha \beta} D_\gamma d \right] \circ f
	\\
	&= m^{ij} \partial_i f^\alpha \partial_j f^\beta \left[ D_\alpha D_\beta d - \tfrac{1}{2} G^{\gamma \kappa} 
		\left( D_\beta G_{\alpha\kappa} + D_\alpha G_{\beta\kappa} - D_\kappa G_{\alpha\beta} \right) D_\gamma d \right] \circ f
	\\
	&= m^{ij} \partial_i f^\alpha \partial_j f^\beta \left[ D_\alpha D_\beta d - D_\gamma d \, D_\beta G_{\alpha\gamma} 
		+ \tfrac{1}{2} D_\gamma d \, D_\gamma G_{\alpha\beta} \right] \circ f	
	\\
	&=	m^{ij} \partial_i f^\alpha \partial_j f^\beta \left[ D_\alpha D_\beta d + D_\gamma d \, D_\beta \left( 2 d \, D_\alpha D_\gamma d - 2 d^2 \, D_\alpha D_\rho d \, D_\gamma D_\rho d \right)
	\right. 
	\\	
	&\left. + D_\gamma d \, D_\gamma \left( - d \, D_\alpha D_\beta d + d^2 D_\alpha D_\rho d \, D_\beta D_\rho d \right) \right] \circ f	
	\\
	&= m^{ij} \partial_i f^\alpha \partial_j f^\beta \left[ D_\alpha D_\beta d + D_\gamma d \left( 2 d \, D_\alpha D_\beta D_\gamma d 
																								- 2 d^2 D_\alpha D_\rho d \, D_\beta D_\gamma D_\rho d \right)
	\right. 
	\\	
	&\left. + D_\gamma d \left( -  D_\gamma d \, D_\alpha D_\beta d - d \, D_\alpha D_\beta D_\gamma d 
		+ 2d \, D_\gamma d \, D_\alpha D_\rho d \, D_\beta D_\rho d + 2 d^2 \, D_\gamma D_\alpha D_\rho d \, D_\beta D_\rho d \right) \right] \circ f	
	\\	
	&= m^{ij} \partial_i f^\alpha \partial_j f^\beta \left[ D_\alpha D_\beta d + 2 d \,  D_\gamma d \, D_\alpha D_\beta D_\gamma d 
																								- 2 d^2  D_\gamma d \, D_\alpha D_\rho d \, D_\beta D_\gamma D_\rho d
	\right.	
	\\	
	&\left. -  D_\alpha D_\beta d - d \, D_\gamma d \, D_\alpha D_\beta D_\gamma d 
		+ 2d \, D_\alpha D_\rho d \, D_\beta D_\rho d + 2 d^2 \, D_\gamma d \, D_\beta D_\rho d \, D_\gamma D_\alpha D_\rho d\right] \circ f		
	\\
	&= m^{ij} \partial_i f^\alpha \partial_j f^\beta \left[ d \,  D_\gamma d \, D_\alpha D_\beta D_\gamma d 
		+ 2d \, D_\alpha D_\rho d \, D_\beta D_\rho d \right] \circ f.	
\end{align*}
Differentiating $D_\gamma d \, D_\beta D_\gamma d = 0$ with respect to the $\alpha$-component gives 
$$
	D_\alpha D_\gamma d \, D_\beta D_\gamma d = - D_\gamma d \, D_\alpha D_\beta D_\gamma d,	
$$
and hence,
$$
	m^{ij} \partial_i f^\alpha \partial_j f^\beta (\nabla^G_\alpha \nabla^G_\beta d) \circ f
	= m^{ij} \partial_i f^\alpha \partial_j f^\beta \, d \circ f \left( D_\alpha D_\rho d \, D_\beta D_\rho d \right) \circ f	.	
$$
\end{proof}

\begin{lem}
\label{b_decomposition}
Let $b: (H^1(\Gamma))^2 \times (H^1(\Gamma))^2 \rightarrow \mathbb{R}$ be defined as in (\ref{b_definition}), where $f: \Gamma \rightarrow \mathbb{S}^1 \subset \mathbb{R}^2$
is a smooth harmonic map on the closed hypersurface $\Gamma \subset \mathbb{R}^{d+1}$, then
$$
	b(\psi, \psi) = \int_\Gamma |\grad \psi_\nu|^2 + |\grad \psi_\tau|^2 \, do + 2 \int_\Gamma |\grad f|^2 \psi_\nu^2 \, do
$$
for all $\psi \in (H^1(\Gamma))^2$.
\end{lem}
\begin{proof}
We plugin the decomposition $\psi = \psi_\nu f + \psi_\tau f^\perp$ into $(\ref{b_definition})$ and obtain
\begin{align*}
b(\psi, \psi) &= \int_\Gamma \D \alpha \left( \psi_\nu f_\beta + \psi_\tau f_\beta^\perp \right) \D \alpha \left( \psi_\nu f_\beta + \psi_\tau f^\perp_\beta \right)
	- 4 \D \alpha f_\beta \D \alpha \left( \psi_\nu f_\beta + \psi_\tau f_\beta^\perp \right) \psi_\nu \, do 
\\ & \quad
	- \int_\Gamma |\grad f|^2 \left( |\psi_\nu|^2 + |\psi_\tau|^2 \right) - 6 |\grad f|^2 |\psi_\nu|^2 \, do
\\ &= \int_\Gamma |\grad \psi_\nu |^2 + |\grad \psi_\tau |^2 + 2 \psi_\tau \D \alpha \psi_\nu f_\beta \D \alpha f_\beta^\perp \, do
\\ & \quad - 2 \int_\Gamma \psi_\nu \D \alpha \psi_\tau f_\beta^\perp \D \alpha f_\beta + \psi_\nu \psi_\tau \D \alpha f_\beta^\perp \D \alpha f_\beta - |\grad f|^2 |\psi_\nu|^2 \, do	.
\end{align*}
From the definition of $f^\perp$ we have $f \cdot f^\perp = 0$ and $\grad f : \grad f^\perp = 0$ as well as $f \cdot \D \alpha f^\perp = - f^\perp \cdot \D \alpha f$.
This yields
\begin{align*}
b(\psi, \psi) = \int_\Gamma |\grad \psi_\nu|^2 + |\grad \psi_\tau|^2 - 2 \left( \psi_\nu \D \alpha \psi_\tau + \D \alpha \psi_\nu \psi_\tau \right) f^\perp_\beta \D \alpha f_\beta
+ 2 |\grad f|^2 |\psi_\nu|^2 \, do.
\end{align*}
Integration by parts gives
\begin{align*}
b(\psi, \psi) &= \int_\Gamma |\grad \psi_\nu|^2 + |\grad \psi_\tau|^2 - 2 \D \alpha (\psi_\nu \psi_\tau ) f^\perp_\beta \D \alpha f_\beta
+ 2 |\grad f|^2 |\psi_\nu|^2 \, do
\\
&= \int_\Gamma |\grad \psi_\nu|^2 + |\grad \psi_\tau|^2 + 2 \psi_\nu \psi_\tau \D \alpha ( f^\perp_\beta \D \alpha f_\beta)
+ 2 |\grad f|^2 |\psi_\nu|^2 \, do.
\end{align*}
Furthermore, we have
$$
	\D \alpha ( f^\perp_\beta \D \alpha f_\beta) = \D \alpha f^\perp_\beta \D \alpha f_\beta + f^\perp_\beta \Delta_\Gamma f_\beta = f^\perp_\beta \Delta_\Gamma f_\beta.
$$
Since $f$ is supposed to be a harmonic map into the $1$-sphere, it holds that
$$
	\Delta_\Gamma f_\beta = - |\grad f|^2 f_\beta.
$$
Therefore, $\D \alpha ( f^\perp_\beta \D \alpha f_\beta) = 0$. This proves the claim.
\end{proof}

\begin{lem}
\label{lemma_smallness_of_first_variation}
Let $h_0 > 0$ be sufficiently small and $h \leq h_0$. For a smooth harmonic map $f: \Gamma \rightarrow \mathbb{S}^n \subset \mathbb{R}^{n+1}$
on the closed hypersurface $\Gamma \subset \mathbb{R}^{d+1}$, $d\leq 3$, the following estimate holds 
\begin{equation*}
	|E_h'(I_h f_l)(\psi_h)| \leq C(\Gamma,f,h_0) h \| \psi_h \|_{H^1} \quad \textnormal{for all $\psi_h \in (V_h)^{n+1}$.}
\end{equation*}
\end{lem}
\begin{proof}
Since $f$ is a harmonic map, we have $E'(f)(\psi_h^l) = 0$, and hence,
\begin{align*}
 &|E_h'(I_h f_l)(\psi_h)| = |E_h'(I_h f_l)(\psi_h) - E'(f)(\psi_h^l)| 
\\ 
 &\leq |\int_{\Gamma_h} \gradh I_h f_l : \gradh \psi_h \left( \tfrac{1}{2} + \tfrac{1}{2|I_h f_l|^4} \right) \, do
 		- \int_{\Gamma} \grad f : \grad \psi_h^l \left( \tfrac{1}{2} + \tfrac{1}{2|f|^4} \right) \, do \,|
\\
 & \quad + |\int_{\Gamma_h} |\gradh I_h f_l|^2 \tfrac{I_h f_l \cdot \psi_h}{|I_h f_l|^6} \, do - \int_{\Gamma} | \grad f|^2 \tfrac{f \cdot \psi_h^l}{|f|^6} \, do \,| 		
\\
 & =: A + B. 
\end{align*}
We first recall the following result from \cite{Dz88}
\begin{equation}
	\left( \gradh \psi_h : \gradh \phi_h \right)^l = R_h^l \grad \psi_h^l : \grad \phi_h^l,
	\label{lift_Laplace_Beltrami}
\end{equation}
with $R_h = P (\unit - d \mathcal{H}) P_h (\unit - d \mathcal{H}) P$.
Using the geometric estimates from Proposition \ref{prop_geometric_estimates} and the fact that $|f| = 1$, we obtain that
\begin{align*}
A &\leq |\int_{\Gamma} R_h^l \grad I_h^l f : \grad \psi_h^l \left( \tfrac{1}{2} + \tfrac{1}{2 |I_h^l f|^4} \right) \mu_h^l \, do 
	- \int_{\Gamma} \grad f : \grad \psi_h^l \left( \tfrac{1}{2} + \tfrac{1}{2 |f|^4} \right) \, do \, |
\\
 &\leq \| P - R_h^l \|_{L^\infty} \| \grad f \|_{L^2} \| \grad \psi_h^l \|_{L^2}	 + \| \mu_h^l - 1 \|_{L^\infty} \| R_h^l \|_{L^\infty} \| \grad f \|_{L^2} \| \grad \psi_h^l \|_{L^2}
\\
 &\quad + \| R_h^l \|_{L^\infty} \| \mu_h^l \|_{L^\infty} \| \grad (I_h^l f - f) \|_{L^2} \| \grad \psi_h^l \|_{L^2}  
\\
 &\quad + \tfrac{1}{2} \| R_h^l \|_{L^\infty} \| \mu_h^l \|_{L^\infty} \| \grad I_h^l f \|_{L^2} \| \grad \psi_h^l \|_{L^2} \| 1/|I_h^l f|^4 - 1/|f|^4 \|_{L^\infty} 
\\
 &\leq C(\Gamma, f , h_0) h \| \grad \psi_h^l \|_{L^2} + C(\Gamma, f, h_0) \| \grad \psi_h^l \|_{L^2} \| f - I_h^l f \|_{L^\infty}
\leq C(\Gamma, f , h_0) h \| \grad \psi_h^l \|_{L^2},
\end{align*}
where we have also used that the map $\mathbb{R}^{n+1} \setminus \{ 0 \} \ni x \mapsto 1/|x|^4$ is locally Lipschitz -- more precisely, 
it is Lipschitz on $\mathbb{R}^{n+1} \setminus B_R(0)$ for every $R > 0$.
Note that $| 1 - |I_h^l f || = ||f| - |I_h^l f|| \leq \| f - I_h^l f \|_{L^\infty} \leq C h^2$. In particular, $I_h^l f(p)$ and $f(p)$ are contained in $\mathbb{R}^{n+1} \setminus B_R(0)$
for $R > 0$ sufficiently small.
Furthermore, we made use of $\| \grad I_h^l f \|_{L^\infty} \leq \|\grad f\|_{L^\infty} + Ch^2 \leq C(f,h_0)$.
We treat the second term in the same way,
\begin{align*}
B &= | \int_\Gamma R_h^l \grad I_h^l f : \grad I_h^l f \tfrac{I_h^l f \cdot \psi_h^l}{|I_h^l f|^6} \mu_h^l \, do - \int_{\Gamma} \grad f : \grad f \tfrac{f \cdot \psi_h^l}{|f|^6} \, do \, |
\\
&= \| P - R_h^l \|_{L^\infty} \| \grad f \|_{L^4}^2 \| \psi_h^l \|_{L^2} + \| \mu_h^l - 1 \|_{L^\infty} \| R_h^l \|_{L^\infty} \| \grad f \|_{L^4}^2 \| \psi_h^l \|_{L^2}
\\
 &\quad + \| \mu_h^l \|_{L^\infty} \| R_h^l \|_{L^\infty} ( \| \grad f \|_{L^\infty} + \| \grad I_h^l f \|_{L^\infty}) \| \grad (f - I_h^l f ) \|_{L^2} \| \psi_h^l \|_{L^2}
\\
 &\quad + \| \mu_h^l \|_{L^\infty} \| R_h^l \|_{L^\infty} \| \grad I_h^l f \|_{L^4}^2 \| I_h^l f /|I_h^l f|^6 - f/|f|^6 \|_{L^\infty} \| \psi_h^l \|_{L^2}
\\
 &\leq  C(\Gamma, f, h_0) h \| \psi_h^l \|_{L^2} + C(\Gamma, f, h_0) \| f - I_h^l f \|_{L^\infty} \| \psi_h^l \|_{L^2}
\leq  C(\Gamma, f, h_0) h \| \psi_h^l \|_{L^2} ,
\end{align*}
where we have used that $\mathbb{R}^{n+1} \setminus \{ 0 \} \ni x \mapsto x/|x|^6$ is locally Lipschitz. Using the equivalence of norms in Proposition \ref{prop_equivalence_of_norms}
the claim follows.
\end{proof}

\begin{lem}
\label{lemma_continuity_second_variation}
Let $h_0 > 0$ be sufficiently small and $h \leq h_0$. Then for a $C^2$-map $f: \Gamma \rightarrow \mathbb{S}^n \subset \mathbb{R}^{n+1}$ 
on the closed hypersurface $\Gamma \subset \mathbb{R}^{d+1}$, $d \leq 3$, it holds that
$$
	|(E''_h (I_h f_l) - E''_h(I_h f_l + \eta_h))(\psi_h, \psi_h)| \leq C(f, h_0) \theta_d(h) \| \eta_h \|_{H^1} \| \psi_h \|^2_{H^1} \sum_{k=0}^4 \theta_d(h)^k \| \eta_h \|_{H^1}^k,
$$
for all $\eta_h, \psi_h \in (V_h)^{n+1}$ with $\| \eta_h \|_{L^\infty} \leq \omega_0$ for some constant $\omega_0 = \omega_0(f,h_0) > 0$.
\end{lem}
\begin{proof}
Using (\ref{second_variation_discrete_energy}), we obtain that
\begin{align*}
& | (E''_h(I_h f_l) - E''_h(I_h f_l + \eta_h))(\psi_h, \psi_h) | \leq \int_{\Gamma_h} | \gradh \psi_h |^2 \Big| \tfrac{1}{2 |I_h f_l|^4} - \tfrac{1}{2 |I_h f_l + \eta_h |^4} \Big| \, do
\\
& \quad + 4 \int_{\Gamma_h} \Big| \gradh I_h f_l : \gradh \psi_h \tfrac{I_h f_l \cdot \psi_h}{|I_h f_l|^6} 
			- \gradh (I_h f_l + \eta_h) : \gradh \psi_h \tfrac{(I_h f_l + \eta_h ) \cdot \psi_h }{|I_h f_l + \eta_h |^6} \Big| \, do
\\
& \quad + \int_{\Gamma_h} \Big| |\gradh I_h f_l|^2 \tfrac{|\psi_h|^2}{|I_h f_l|^6} - |\gradh (I_h f_l + \eta_h)|^2 \tfrac{|\psi_h|^2}{|I_h f_l + \eta_h|^6} \Big| \, do		
\\
& \quad + 6 \int_{\Gamma_h} \Big| |\gradh I_h f_l|^2 \tfrac{[I_h f_l \cdot \psi_h ]^2}{|I_h f_l|^8} 
								- |\gradh (I_h f_l + \eta_h)|^2 \tfrac{[(I_h f_l + \eta_h ) \cdot \psi_h ]^2}{|I_h f_l + \eta_h |^8} \Big| \, do
\\
&=: T_1 + T_2 + T_3 + T_4.	
\end{align*}
For the first term the local Lipschitz continuity of $\mathbb{R}^{n+1} \setminus \{ 0 \} \ni x \mapsto 1/|x|^4$ gives
\begin{align*}
 T_1 &\leq \tfrac{1}{2} \| \tfrac{1}{|I_h f_l|^4} - \tfrac{1}{|I_h f_l + \eta_h |^4} \|_{L^\infty} \| \gradh \psi_h \|_{L^2}^2
 \leq C(f, h_0) \| \eta_h \|_{L^\infty} \| \gradh \psi_h \|_{L^2}^2
 \\
 &\leq C(f, h_0) \theta_d(h) \| \eta_h \|_{H^1} \| \gradh \psi_h \|_{L^2}^2, 
\end{align*}
where we have made use of Lemma \ref{lemma_bartels} and the fact that $(I_h f_l + \eta_h)(p) \in \mathbb{R}^{n+1}\setminus B_R(0)$ for some $R > 0$
if $\omega_0$ is sufficiently small. Similarly, we obtain
\begin{align*}
T_2 &\leq 4 \int_{\Gamma_h} \Big| \gradh \eta_h : \gradh \psi_h \tfrac{I_h f_l \cdot \psi_h }{|I_h f_l|^6} \Big| 
+ \Big| \gradh (I_h f_l + \eta_h) : \gradh \psi_h \left( \tfrac{I_h f_l}{ |I_h f_l|^6} - \tfrac{ I_h f_l + \eta_h}{|I_h f_l + \eta_h|^6} \right) \cdot \psi_h \Big| \, do
\\
&\leq C(f, h_0) \| \gradh \eta_h \|_{L^2} \| \gradh \psi_h \|_{L^2} \| \psi_h \|_{L^\infty}
+ C(f, h_0) \| \gradh \psi_h \|_{L^2}  \| \eta_h \|_{L^\infty} \| \psi_h \|_{L^2} 
\\
&\quad + C(f, h_0)  \| \gradh \eta_h \|_{L^2} \| \gradh \psi_h \|_{L^2}  \| \eta_h \|_{L^\infty} \| \psi_h \|_{L^\infty} 
\\
&\leq C(f, h_0) \theta_d(h) \| \gradh \eta_h \|_{L^2} \| \gradh \psi_h \|_{L^2} \| \psi_h \|_{H^1}
+ C(f, h_0) \theta_d(h) \| \eta_h \|_{H^1} \| \gradh \psi_h \|_{L^2} \| \psi_h \|_{L^2}
\\
&\quad + C(f, h_0) \theta_d(h)^2 \| \gradh \eta_h \|_{L^2} \| \gradh \psi_h \|_{L^2} \| \eta_h \|_{H^1} \| \psi_h \|_{H^1}
\\
&\leq C(f, h_0) (1 + \theta_d(h) \| \eta_h \|_{H^1}) \theta_d(h) \| \eta_h \|_{H^1} \| \psi_h \|_{H^1}^2.
\end{align*}
For the third term a Lipschitz argument together with Lemma \ref{lemma_bartels} leads to
\begin{align*}
T_3 &= C(f,h_0) \int_{\Gamma_h} \Big| |\gradh I_h f_l|^2 - |\gradh (I_h f_l + \eta_h)|^2 \Big| |\psi_h|^2 \, do
\\
&\quad + \int_{\Gamma_h} | \gradh (I_h f_l + \eta_h)|^2 \Big| \tfrac{1}{|I_h f_l|^6} - \tfrac{1}{|I_h f_l + \eta_h|^6} \Big| |\psi_h|^2 \, do
\\
&\leq C(f,h_0) \int_{\Gamma_h} (1 + | \gradh \eta_h |) | \gradh \eta_h | |\psi_h|^2 \, do
	+ C(f,h_0) \int_{\Gamma_h} (1 + | \gradh \eta_h |^2) | \eta_h | |\psi_h|^2 \, do
\\
&\leq C(f,h_0) \left(
	\| \gradh \eta_h \|_{L^2} \| \psi_h \|_{L^2} \| \psi_h \|_{L^\infty}
	+ \| \gradh \eta_h \|_{L^2}^2 \| \psi_h \|_{L^\infty}^2
	+ \| \eta_h \|_{L^\infty} \| \psi_h \|_{L^2}^2 \right.
\\
&\qquad \qquad \qquad	
\left.
	+ \| \gradh \eta_h \|_{L^2}^2 \| \eta_h \|_{L^\infty} \| \psi_h \|_{L^\infty}^2
\right)
\\
&\leq C(f,h_0) \left( \theta_d(h) \| \grad \eta_h \|_{L^2} \| \psi_h \|_{L^2} \| \psi_h \|_{H^1}
	+ \theta_d(h)^2 \| \grad \eta_h \|_{L^2}^2 \| \psi_h \|_{H^1}^2 \right.
\\
&\qquad \qquad \qquad	
\left.
	+ \theta_d(h) \| \eta_h \|_{H^1} \| \psi_h \|_{L^2}^2
	+ \theta_d(h)^3 \| \eta_h \|_{H^1}^3 \| \psi_h \|_{H^1}^2
\right).
\end{align*}
Finally, the local Lipschitz continuity of $\mathbb{R}^{n+1} \setminus \{ 0 \} \ni x \mapsto 1/|x|^8$ and Lemma \ref{lemma_bartels} gives
\begin{align*}
T_4 &\leq C(f,h_0) \int_{\Gamma_h} \Big| |\gradh I_h f_l|^2 - |\gradh (I_h f_l + \eta_h)|^2 \Big| |\psi_h|^2 \, do
\\
&\quad + \int_{\Gamma_h} |\gradh (I_h f_l + \eta_h)|^2 \Big| \tfrac{[I_h f_l \cdot \psi_h]^2}{|I_h f_l|^8} 
							- \tfrac{[I_h f_l \cdot \psi_h + \eta_h \cdot \psi_h]^2}{|I_h f_l + \eta_h|^8} \Big| \, do
\\
&\leq C(f,h_0) \left( \int_{\Gamma_h} (1 + |\gradh \eta_h|) |\gradh \eta_h| |\psi_h|^2 + (1 + |\gradh \eta_h|^2) |\eta_h| [I_h f_l \cdot \psi_h + \eta_h \cdot \psi_h ]^2 \, do
\right.
\\
&\quad 
\left.
+ \int_{\Gamma_h} (1 + |\gradh \eta_h|^2) |[I_h f_l \cdot \psi_h ]^2 - [I_h f_l \cdot \psi_h + \eta_h \cdot \psi_h]^2| \, do							
\right)
\\
&\leq C(f,h_0) \left( \| \gradh \eta_h \|_{L^2} \| \psi_h \|_{L^2} \| \psi_h \|_{L^\infty} + \| \gradh \eta_h \|_{L^2}^2 \| \psi_h \|_{L^\infty}^2
\right)
\\
&\quad
+ C(f,h_0) \int_{\Gamma_h} (1 + |\gradh \eta_h|^2) |\eta_h| |\psi_h|^2 (1 + |\eta_h| + |\eta_h|^2) \, do
\\
&\leq C(f,h_0) \left( \theta_d(h) \| \gradh \eta_h \|_{L^2} \| \psi_h \|_{L^2} \| \psi_h \|_{H^1} + \theta_d(h)^2 \| \gradh \eta_h \|_{L^2}^2 \| \psi_h \|_{H^1}^2
\right)
\\
&\quad
+ C(f,h_0) 
( \| \psi_h \|_{L^2}^2 + \theta_d(h)^2 \| \gradh \eta_h \|_{L^2}^2  \| \psi_h \|_{H^1}^2 )
\left(
	\sum_{k=1}^3 \theta_d(h)^k \| \eta_h \|_{H^1}^k
\right).
\end{align*}
\end{proof}

\begin{lem}
\label{lemma_consistency_second_variation}
Let $h_0 > 0$ be sufficiently small and $h \leq h_0$. Then for a $C^2$-map $f: \Gamma \rightarrow \mathbb{S}^n \subset \mathbb{R}^{n+1}$
on the closed hypersurface $\Gamma \subset \mathbb{R}^{d+1}$, $d \leq 3$, the estimate
\begin{align*}
	|E''(f)(\psi_h^l, \psi_h^l) - E''_h(I_h f_l)(\psi_h, \psi_h)| \leq C(\Gamma,f,h_0) h \| \psi_h \|_{H^1}^2
\end{align*}
holds for all $\psi_h \in (V_h)^{n+1}$.
\end{lem}
\begin{proof}
We insert (\ref{second_variation_discrete_energy}) and (\ref{second_variation_continuous_case}) into
\begin{align*}
	|E''(f)(\psi_h^l, \psi_h^l) - E''_h(I_h f_l)(\psi_h, \psi_h)| \leq I_1 + 4 I_2 + I_3 + 6 I_4,
\end{align*}
where the terms $I_1, I_2, I_3$ and $I_4$ are defined below. Using (\ref{lift_Laplace_Beltrami}), the geometric estimates from Proposition \ref{prop_geometric_estimates}
and the fact that $\mathbb{R}^{n+1} \setminus \{ 0 \} \ni x \mapsto 1/|x|^4$ is locally Lipschitz -- recall that $|f| = 1$ and $|1 - |I_h f_l|| \leq Ch^2$, we obtain for the first term
\begin{align*}
I_1 &:= \Big| \int_\Gamma |\grad \psi_h^l|^2 \left( \tfrac{1}{2} + \tfrac{1}{2|f|^4} \right) \, do 
		- \int_{\Gamma_h} |\gradh \psi_h|^2 \left( \tfrac{1}{2} + \tfrac{1}{2|I_h f_l|^4} \right) \, do \Big|
\\
	&= 	\Big| \int_\Gamma |\grad \psi_h^l|^2 \left( \tfrac{1}{2} + \tfrac{1}{2|f|^4} \right) \, do 
		- \int_\Gamma R_h^l \grad \psi_h^l : \grad \psi_h^l \left( \tfrac{1}{2} + \tfrac{1}{2|I_h^l f|^4} \right) \mu_h^l \, do \Big|	
\\
	&\leq \Big| \int_\Gamma (P - R_h^l) \grad \psi_h^l : \grad \psi_h^l \left( \tfrac{1}{2} + \tfrac{1}{2|f|^4} \right) \, do \Big|		
	+ \Big| \int_\Gamma R_h^l \grad \psi_h^l : \grad \psi_h^l \left( \tfrac{1}{2} + \tfrac{1}{2|f|^4} \right) (1 - \mu_h^l) \, do \Big|
\\
	&\quad + \Big| \int_\Gamma R_h^l \grad \psi_h^l : \grad \psi_h^l \left( \tfrac{1}{2|f|^4} - \tfrac{1}{2|I_h^l f|^4} \right) \mu_h^l \, do \Big|
\\
	&\leq \| P - R_h^l \|_{L^\infty} \| \grad \psi_h^l \|_{L^2}^2	 + \| R_h^l \|_{L^\infty} \| 1 - \mu_h^l \|_{L^\infty} \| \grad \psi_h^l \|_{L^2}^2
\\	
	&\quad + C(f,h_0) \| R_h^l \|_{L^\infty} \| \mu_h^l \|_{L^\infty} \| \grad \psi_h^l \|_{L^2}^2 \| f - I_h^l f \|_{L^\infty}
\\
	&\leq C(\Gamma,f,h_0) h^2 \| \grad \psi_h^l \|_{L^2}^2.	
\end{align*}
Similarly, we conclude that 
\begin{align*}
I_2 &:= \Big| \int_\Gamma \grad f : \grad \psi_h^l \tfrac{f \cdot \psi_h^l}{|f|^6} \, do - \int_{\Gamma_h} \gradh I_h f_l : \gradh \psi_h \tfrac{I_h f_l \cdot \psi_h}{|I_h f_l|^6} \, do \Big|
\\
	&= \Big| \int_\Gamma \grad f : \grad \psi_h^l \tfrac{f \cdot \psi_h^l}{|f|^6} \, do - \int_\Gamma R_h^l \grad I_h^l f : \grad \psi_h^l \tfrac{I_h^l f \cdot \psi_h^l}{|I_h^l f|^6} \mu_h^l \, do \Big|
\\
	&= \Big| \int_\Gamma (P - R_h^l) \grad f : \grad \psi_h^l \tfrac{f \cdot \psi_h^l}{|f|^6} \, do \Big|	
	+ \Big| \int_\Gamma R_h^l \grad f : \grad \psi_h^l \tfrac{f \cdot \psi_h^l}{|f|^6} (1 - \mu_h^l) \, do \Big|	
\\
	&\quad + \Big| \int_\Gamma R_h^l \grad (f - I_h^l f) : \grad \psi_h^l \tfrac{f \cdot \psi_h^l}{|f|^6} \mu_h^l \, do \Big|		
	+ \Big| \int_\Gamma R_h^l \grad I_h^l f : \grad \psi_h^l \left( \tfrac{f}{|f|^6} - \tfrac{I_h^l f}{|I_h^l f|^6} \right) \cdot \psi_h^l \mu_h^l \, do \Big|	
\end{align*}
\begin{align*} 
I_2	&\leq \| P - R_h^l \|_{L^\infty} \| \grad f \|_{L^\infty}	 \| \grad \psi_h^l \|_{L^2} \| \psi_h^l \|_{L^2}
	+ \| R_h^l \|_{L^\infty} \| 1 - \mu_h^l \|_{L^\infty} \| \grad f \|_{L^\infty} \| \grad \psi_h^l \|_{L^2} \| \psi_h^l \|_{L^2}
\\
	&\quad + \| R_h^l \|_{L^\infty}	\| \mu_h^l \|_{L^\infty} \| \grad (f - I_h^l f) \|_{L^\infty} \| \grad \psi_h^l \|_{L^2} \| \psi_h^l \|_{L^2}
\\
	&\quad + C(f,h_0) \| R_h^l \|_{L^\infty}	\| \mu_h^l \|_{L^\infty} \| \grad I_h^l f \|_{L^\infty} \| \grad \psi_h^l \|_{L^2} \| \psi_h^l \|_{L^2} \| f - I_h^l f \|_{L^\infty}
\\
	&\leq C(\Gamma, f, h_0) h \| \grad \psi_h^l \|_{L^2} \| \psi_h^l \|_{L^2}	,
\end{align*}
and that
\begin{align*}
I_3 &:= \Big| \int_\Gamma |\grad f|^2 \tfrac{|\psi_h^l |^2}{|f|^6} \, do - \int_{\Gamma_h} |\gradh I_h f_l|^2 \tfrac{|\psi_h|^2}{|I_h f_l|^6} \, do \Big|
\\
&= \Big| \int_\Gamma |\grad f|^2 \tfrac{|\psi_h^l |^2}{|f|^6} \, do - \int_\Gamma R_h^l \grad I_h^l f : \grad I_h^l f \tfrac{|\psi_h^l|^2}{|I_h^l f|^6} \mu_h^l \, do \Big|
\\
&\leq \Big| \int_\Gamma (P - R_h^l) \grad f : \grad f \tfrac{|\psi_h^l|^2}{|f|^6} \, do \Big|
+ \Big| \int_\Gamma R_h^l \grad f : \grad f \tfrac{|\psi_h^l|^2}{|f|^6} (1 - \mu_h^l) \, do \Big|
\\
&\quad + \Big| \int_\Gamma R_h^l \grad (f - I_h^l f) : \grad f \tfrac{|\psi_h^l|^2}{|f|^6} \mu_h^l \, do \Big|
+ \Big| \int_\Gamma R_h^l \grad I_h^l f : \grad (f - I_h^l f) \tfrac{|\psi_h^l|^2}{|f|^6} \mu_h^l \, do \Big|
\\
&\quad + \Big| \int_\Gamma R_h^l \grad I_h^l f : \grad I_h^l f \left( \tfrac{1}{|f|^6} - \tfrac{1}{|I_h^l f|^6} \right) |\psi_h^l|^2 \mu_h^l \, do \Big|
\\
&\leq \| P - R_h^l \|_{L^\infty} \| \grad f \|_{L^\infty}^2 \| \psi_h^l \|_{L^2}^2
+ \| R_h^l \|_{L^\infty} \| 1 - \mu_h^l \|_{L^\infty} \| \grad f \|_{L^\infty}^2 \| \psi_h^l \|_{L^2}^2
\\
&\quad + \| R_h^l \|_{L^\infty} \| \mu_h^l \|_{L^\infty} \| \grad (f - I_h^l f) \|_{L^\infty} ( \| \grad f \|_{L^\infty} + \| \grad I_h^l f \|_{L^\infty}) \| \psi_h^l \|_{L^2}^2
\\
&\quad + C(f,h_0) \| R_h^l \|_{L^\infty} \| \mu_h^l \|_{L^\infty} \| \grad I_h^l f \|_{L^\infty}^2  \| f - I_h^l f \|_{L^\infty} \| \psi_h^l \|_{L^2}^2
\\
&\leq C(\Gamma,f,h_0) h \| \psi_h^l \|_{L^2}^2.
\end{align*}
Finally, we obtain
\begin{align*}
I_4 &:= \Big| \int_\Gamma |\grad f|^2 \tfrac{[f \cdot \psi_h^l]^2}{|f|^8} \, do - \int_{\Gamma_h} |\gradh I_h f_l|^2 \tfrac{[I_h f_l \cdot \psi_h]^2}{|I_h f_l|^8} \, do \Big|
\\
&= \Big| \int_\Gamma |\grad f|^2 \tfrac{[f \cdot \psi_h^l]^2}{|f|^8} \, do 
	- \int_{\Gamma_h} R_h^l \grad I_h^l f : \grad I_h^l f \tfrac{[I_h^l f \cdot \psi_h^l]^2}{|I_h^l f|^8} \mu_h^l \, do \Big|
\\
&= \Big| \int_\Gamma (P - R_h^l) \grad f : \grad f \tfrac{[f \cdot \psi_h^l]^2}{|f|^8} \, do \Big|	
+ \Big| \int_\Gamma R_h^l \grad f : \grad f \tfrac{[f \cdot \psi_h^l ]^2}{|f|^8} (1 - \mu_h^l) \, do \Big|
\\
&\quad + \Big| \int_\Gamma R_h^l \grad (f - I_h^l f) : \grad f \tfrac{[f \cdot \psi_h^l]^2}{|f|^8}  \mu_h^l \, do \Big|
+ \Big| \int_\Gamma R_h^l \grad I_h^l f : \grad (f - I_h^l f) \tfrac{[f \cdot \psi_h^l]^2}{|f|^8} \mu_h^l \, do \Big|
\\
&\quad + \Big| \int_\Gamma R_h^l \grad I_h^l f : \grad I_h^l f  \left( \tfrac{[f \cdot \psi_h^l]^2}{|f|^8} - \tfrac{[I_h^l f \cdot \psi_h^l]^2}{|I_h^l f|^8} \right) \mu_h^l \, do \Big|
\\
&\leq \| P - R_h^l \|_{L^\infty} \| \grad f \|_{L^\infty}^2 \| \psi_h^l \|_{L^2}^2
+ \| R_h^l \|_{L^\infty} \| 1 - \mu_h^l \|_{L^\infty} \| \grad f \|_{L^\infty}^2 \| \psi_h^l \|_{L^2}^2
\\
&\quad + \| R_h^l \|_{L^\infty} \| \mu_h^l \|_{L^\infty} \| \grad (f - I_h^l f) \|_{L^\infty} (\| \grad f \|_{L^\infty} + \| \grad I_h^l f \|_{L^\infty}) \| \psi_h^l \|_{L^2}^2
\\
&\quad + \| R_h^l \|_{L^\infty} \| \mu_h^l \|_{L^\infty} \| \grad I_h^l f \|_{L^\infty}^2
\int_\Gamma \Big| \tfrac{[f \cdot \psi_h^l]^2}{|f|^8} - \tfrac{[I_h^l f \cdot \psi_h^l]^2}{|I_h^l f|^8} \Big| \, do
\\
&\leq C(\Gamma, f, h_0) \left( h \| \psi_h^l \|_{L^2}^2 + \int_\Gamma \Big| \tfrac{(f - I_h^l f) \cdot \psi_h^l (\psi_h^l \cdot f)}{|f|^8} \Big| \, do
+ \int_\Gamma \Big| \tfrac{(I_h^l f \cdot \psi_h^l) \psi_h^l \cdot (f - I_h^l f)}{|f|^8} \Big| \, do
\right.
\\
&\quad 
\left.
+ \int_\Gamma \Big| [I_h^l f \cdot \psi_h^l]^2 \left( \tfrac{1}{|f|^8} - \tfrac{1}{|I_h^l f|^8} \right) \Big| \, do
\right)
\\
&\leq C(\Gamma, f, h_0) \left( h + \| f - I_h^l f \|_{L^\infty} \right) \| \psi_h^l \|_{L^2}^2
\leq C(\Gamma, f, h_0) h \| \psi_h^l \|_{L^2}^2.
\end{align*}
It follows that 
\begin{align*}
|E''(f)(\psi_h^l, \psi_h^l) - E''_h(I_h f_l)(\psi_h, \psi_h)| &\leq C(\Gamma, f, h_0) h \left( \| \grad \psi_h^l \|_{L^2}^2 +
\| \grad \psi_h^l \|_{L^2} \| \psi_h^l \|_{L^2} + \| \psi_h^l \|_{L^2}^2 \right)
\\
&\leq C(\Gamma, f, h_0) h \left( \| \grad \psi_h^l \|_{L^2}^2 + \| \psi_h^l \|_{L^2}^2 \right) 
\\
&= C(\Gamma, f, h_0) h \| \psi_h^l \|_{H^1}^2.
\end{align*}
The equivalence of norms in Proposition \ref{prop_equivalence_of_norms} gives the result.
\end{proof}

\subsection*{Acknowledgement}
We thank Gerhard Dziuk and Jan Steinhilber to call our attention to the problem of computing discrete harmonic maps.
We thank Harald Garcke for discussions.


\begin{thebibliography}{9}
\bibitem{Al97}
  F. Alouges,
  \emph{{A new algorithm for computing liquid crystal stable configurations: The harmonic map case}},
  SIAM J. Numer. Anal. Vol. \textbf{35}, No. 5 (1997), 1708--1726.

\bibitem{Ba05}
  S. Bartels,
  \emph{{Stability and convergence of finite-element approximation schemes for harmonic maps}},
  SIAM J. Numer. Anal. \textbf{43}, No. 1 (2005), 220--238.

\bibitem{Ba08}
  S. Bartels,
  \emph{{Finite Element Approximation of Harmonic Maps between Surfaces}},
  Habilitation treatise, Humboldt University of Berlin (2008).
  URL \url{https://portal.uni-freiburg.de/aam/abtlg/ls/lsbartels/publ/thes/bart09-thesis.pdf}

\bibitem{Ba10}
  S. Bartels,
  \emph{{Numerical analysis of a finite element scheme for the approximation of harmonic maps into surfaces}},
  Math. Comp. \textbf{79}, No. 271 (2010),1263--1301.
  
\bibitem{Ba15}
  S. Bartels,
  \emph{{Numerical methods for nonlinear partial differential equations}},
  Springer Series in Computational Mathematics \textbf{47}, Springer Cham (2015).  

\bibitem{COV04}
 T. Cecil, S. Osher and L. Vese,
 \emph{{Numerical methods for minimization problems constrained to $\mathbb{S}^1$ and $\mathbb{S}^2$}},
 J. of Computational Physics \textbf{198} (2004), 567--579.
  
\bibitem{CL95}
 Y. Chen and F. H. Lin,
 \emph{{Remarks on approximate harmonic maps}},  
 Comment. Math. Helvetici \textbf{70} (1995), 161--169.  
  
\bibitem{CD03}
  U. Clarenz and G. Dziuk, 
  \emph{{Numerical methods for conformally parametrized surfaces}},
  CPDw04 - Interphase 2003: Numerical Methods for Free Boundary Problems (2003).
  URL \url{http://www.newton.ac.uk/webseminars/pg+ws/2003/cpd/cpdw04/0415/dziuk}

\bibitem{CHKL87}
  R. Cohen, R. Hardt, D. Kinderlehrer, S.-Y. Lin and M. Luskin,
  \emph{{Minimum energy configurations for liquid crystals: Computational results}},
  Theory and Applications of Liquid Crystals, IMA Vol. \textbf{5}, Springer New York (1987), 99--122.

\bibitem{DDE05}
	K. Deckelnick, G. Dziuk and C. M. Elliott,  
	\emph{{Computation of geometric partial differential equations and mean curvature flow}}, 
	Acta Numerica \textbf{14} (2005), 139--232.

\bibitem{Dem09}
	A. Demlow, 
	\emph{{Higher-order finite element methods and pointwise error estimates for elliptic problems on surfaces}}, 
	SIAM J. Numer. Anal. \textbf{47}, No. 2 (2009), 805--827.


\bibitem{Dz88}
	G. Dziuk,
	\emph{{Finite Elements for the Beltrami operator on arbitrary surfaces}},
	In S. Hildebrandt, R. Leis (eds.): Partial differential equations and calculus of variations, Springer
	Lecture Notes in Mathematics 1357 (1988), 142--155.
	
\bibitem{DH99}
	G. Dziuk and J. E. Hutchinson,
	\emph{{The discrete Plateau problem: Convergence results}},
	Math. of Computation Vol. \textbf{68}, No. 226 (1999), 519--546.

\bibitem{ES64}
  J. Eells and J. H. Sampson,  \emph{{Harmonic mappings of Riemannian manifolds}},
  Amer. J. Math. \textbf{86} (1964), 109--160.


\bibitem{Fr15}
  H. Fritz, \emph{{Numerical Ricci-DeTurck Flow}}, Numerische Mathematik \textbf{131}, No. 2 (2015), 
  241--271.

\bibitem{Ham75}  
 R.~S. Hamilton, \emph{{Harmonic maps of manifolds with boundary}},
 Springer Lecture Notes \textbf{471} (1975).

\bibitem{He04}
 C.-J. Heine, \emph{{Isoparametric finite element approximation of curvature on hypersurfaces}},
 Preprint Fak. f. Math. Phys. University of Freiburg (2004).

\bibitem{HTW09}  
	Q. Hu, X.-C. Tai and R. Winther,
	\emph{{A saddle point approach to the computation of harmonic maps}}, 
	SIAM J. Numer. Anal. Vol. \textbf{47}, No. 2 (2009), 1500--1523.  

\bibitem{Hu94}
	N. Hungerb\"uhler,
	\emph{{p-harmonic Flow}},	
	PhD thesis, ETH Z\"urich (1994).
	URL \url{http://www.math.ch/norbert.hungerbuehler/publications/diss/d.pdf}
	
\bibitem{Kl82}
 W. Klingenberg, \emph{{Riemannian Geometry}},
 de Gruyter Studies in Mathematics 1, de Gruyter \& Co. Berlin (1982).

\bibitem{ON83}
 B. O'Neill, \emph{{Semi-Riemannian Geometry}},
 Academic Press San Diego (1983).

\bibitem{OV02}
 S. Osher and L. A. Vese, 
 \emph{{Numerical methods for $p$-harmonic flows and applications to image processing}},
 SIAM J. Numer. Anal. Vol. \textbf{40}, No. 6 (2002), 2085--2104.

\bibitem{Sa12}
 O. Sander, \emph{{Geodesic finite elements on simplicial grids}},
 Int. J. Num. Meth. Eng. 92(12):999--1025 (2012).  
  
\bibitem{Sa15}  
  O. Sander, \emph{{Geodesic finite elements of higher order}}, 
  IMA J. Numer. Anal. (2015). doi: 10.1093/imanum/drv016.   
  

\bibitem{St14} 
  J. Steinhilber, \emph{{Numerical analysis for harmonic maps between hypersurfaces 	
  and grid improvement for computational parametric geometric flows}},
  PhD thesis, University of Freiburg (2014). 
  URL \url{http://www.freidok.uni-freiburg.de/volltexte/9537/}

\bibitem{SS05}
  A. Schmidt and K. G. Siebert, \emph{{Design of Adaptive Finite Element Software}},	
  Lecture Notes in Computational Science and Engineering \textbf{42}, Springer (2005).  
  
\bibitem{St98}
  M. Struwe, \emph{{Uniqueness of harmonic maps with small energy}},
  Manuscripta Math. \textbf{96} (1998), 463--486.  
  
\bibitem{Top06}
  P. Topping, \emph{{Lectures on the Ricci Flow}},
  London Mathematical Society Lecture Note Series: \textbf{325}, Cambridge University Press (2006).  
\end{thebibliography}
\end{document}